\documentclass[10pt, a4paper]{article}

\usepackage{amsthm,amssymb,mathrsfs,setspace,amsmath}

\usepackage{float}
\usepackage{color}
\usepackage[top=1in, bottom=1in, left=1in, right=1in]{geometry}

\usepackage{graphicx}
\usepackage{subfigure}

\numberwithin{figure}{section}

\newtheorem{lemma}{Lemma}[section]

\begin{document}

\title{Numerical solutions for a class of singular boundary value problems arising in the theory of epitaxial growth }
\author{Amit Kumar Verma$^a$, Biswajit Pandit$^b$\thanks{Email:$^a$akverma@iitp.ac.in,$^b$biswajitpandit82@gmail.com}
\\\small{\textit{$^{a,b}$Department of Mathematics, Indian Institute of Technology Patna, Patna--$801106$, Bihar, India.}}\\\\
Carlos Escudero$^c$\thanks{Email:$^{c}$cel@icmat.es}
\\\small{\textit{$^{c}$Departamento de Matem{\'a}ticas $\&$ ICMAT (CSIC-UAM-UC3M-UCM),}}\\
\small\textit{Universidad Aut{\'o}noma de Madrid, E-28049 Madrid, Spain.}
}

\date{\today}

\maketitle

\begin{abstract}
The existence of numerical solutions to a fourth order singular boundary value problem arising in the theory of epitaxial growth is studied. An iterative numerical method is applied on a second order nonlinear singular boundary value problem which is the exact result of the reduction of this fourth order singular  boundary value problem. It turns out that the existence or nonexistence of numerical solutions fully depends on the value of a parameter. We show that numerical solutions exist for small positive values of this parameter. For large positive values of the parameter, we find nonexistence of solutions. We also observe existence of solutions for negative values of the parameter and determine the range of parameter values which separates existence and nonexistence of solutions.
This parameter has a clear physical meaning as it describes the rate at which new material is deposited onto the system. This fact allows us to interpret
the physical significance of our results.
\end{abstract}

\section{Introduction}

 Along the years, a revolution of semiconductor device design has been spawned by the invention of super-lattices and similar structures. These  structures  led to an improvement in the performances of many electronic instruments like lasers, diodes and bipolar transistors. Such advanced structures can be produced by means of epitaxial growth techniques. Epitaxial growth techniques that produce thin films under high vacuum conditions (\cite{Barabasi:1995}) have largely superseded older technologies in the semiconductor industry. Due to this success of epitaxial growth techniques, several mathematical models have been introduced in order
 to  understand them better. In general terms, one could say that these  models belong to one of two different classes: they are either discrete probabilistic models, such as cellular automata, or differential equations (\cite{Barabasi:1995}). In this work we strictly focus on the second type of model, and in particular we restrict our attention to the differential equation described in (\cite{Carlos2008,Carlos2012,CarlosOrigin2012,CarlosRadial2013}). In these references the mathematical description of epitaxial growth is carried out by means of a function,
\begin{equation}
\begin{aligned}\label{Eq 1}
&\sigma : \Omega \subset \mathbb{R}^{2}\times \mathbb{R}^{+} \rightarrow \mathbb{R},
\end{aligned}
\end{equation}
which describes the height of the growing interface in the spatial point $x \in \Omega \subset \mathbb{R}^{2} $ at time $t \in \mathbb{R}^{+}$. This function obeys the fourth order partial differential equation (\cite{Carlos2008})
\begin{equation}
\partial_t \sigma +\Delta^{2} \sigma~=~\text{det}(D^{2}\sigma)+\lambda \Gamma(x,t),~~~x\in \Omega \subset \mathbb{R}^{2},
\end{equation}
where $\Gamma(x,t)$ models the incoming mass entering the system through epitaxial deposition and $\lambda$ measures the intensity of this flux.
Roughly speaking, the thin film grows by the introduction of new mass into the system (modeled by $\lambda \Gamma(x,t)$) and its structure is
characterized by means of its height $\sigma(x,t)$. A basic modeling assumption is that this structure is dominated by processes taking place at the
surface of the thin film which are coarse-grained modeled by the biharmonic operator and the nonlinearity.
For simplicity we will focus on the stationary counterpart of this partial differential equation:
\begin{equation}\label{Eq 2}
\Delta^{2} \sigma~=~\text{det}(D^{2}\sigma)+\lambda G(x),~~~x\in \Omega \subset \mathbb{R}^{2},
\end{equation}
where we have assumed that $\Gamma(x,t) \equiv G(x)$ is a stationary flux, and we set this problem on the unit disk again for simplicity. As previously we will consider two types of boundary conditions, namely homogeneous Dirichlet and homogeneous Navier boundary conditions (\cite{Carlos2012}). By using the transformation  $r=|x|$ and $ \sigma(x)=\phi(|x|)$  the above partial differential equation (\ref{Eq 2}) is converted  into a fourth order ordinary differential equation which reads
\begin{equation}
  \begin{aligned}\label{Eq 3}
&\frac{1}{r} \left \{ r\left[\frac{1}{r}(r \phi')^{'}\right]'\right \}^{'}=\frac{1}{r} \phi^{'} \phi^{''}+\lambda G(r)
 \end{aligned}
\end{equation}
where $\displaystyle '=\frac{d}{dr}$.

The Dirichlet boundary conditions that correspond to (\ref{Eq 3}) are 
\begin{eqnarray}
\phi'(0)=0,~\phi(1)=0,~\phi'(1)=0,~\displaystyle \lim_{r\rightarrow0} r\phi'''(r)=0,
\end{eqnarray}
the Navier boundary conditions of type one are
\begin{eqnarray}
\phi'(0)=0,~\phi(1)=0,~\phi'(1)+\phi''(1)=0,~\displaystyle \lim_{r\rightarrow0} r\phi'''(r)=0,
\end{eqnarray}
and the Navier boundary conditions of type two are
\begin{eqnarray}
\phi'(0)=0,~\phi(1)=0,~\phi''(1)=0,~\displaystyle \lim_{r\rightarrow0} r\phi'''(r)=0.
\end{eqnarray}
The condition $\phi'(0)=0$  imposes the existence of an extremum at the origin. The conditions  $\phi(1)=0$ and $\phi'(1)=0$ are the actual boundary conditions. For simplicity we take  $G(r)=1$, which physically means that the new material is being deposited uniformly on the unit disc. Using $\displaystyle \lim_{r\rightarrow0} r\phi'''(r)=0$ and integrating by parts,  equation (\ref{Eq 3}) gives
\begin{equation}
\begin{aligned}\label{Eq 4}
 & r\left[\frac{1}{r}(r \phi')^{'}\right]'=\frac{1}{2}(\phi^{'})^{2}+\frac{1}{2}\lambda  r^{2}.
 \end{aligned}
\end{equation}
Using the second transformation $ w=r \phi'$, from equation (\ref{Eq 4}) we get
\begin{equation}
\begin{aligned}\label{Eq 5}
&&&&r^{2} w''(r)-r w'(r)=\frac{1}{2} w^{2}(r)+\frac{1}{2} \lambda r^{4},
\end{aligned}
\end{equation}
Corresponding to homogeneous Dirichlet boundary condition is 
\begin{equation}
w'(0)=0 ~\mbox{and}~ w(1)=0,
\end{equation} 
homogeneous Navier boundary condition of type one is 
\begin{equation}
w'(0)=0~ \mbox{and} ~w'(1)=0,
\end{equation} 
and homogeneous Navier boundary condition of type two is
\begin{equation}
w'(0)=0 ~\mbox{and}~ w(1)=w'(1).
\end{equation} 
 Equation (\ref{Eq 5}) was numerically integrated by means of the use of a fourth order Runge-Kutta method in \cite{CarlosRadial2013}. Now, we understand the solution of equation (\ref{Eq 5}) belonging to the space $C^{2}\left[0,1\right]$.\nocite{PandeySBVP2008}
 
 Equation (\ref{Eq 5}) is a nonlinear, non self-adjoint and singular differential equation. Further more it has multiple solutions. Therefore discrete methods such as finite element method etc may not be applicable  to pick all solutions together. These facts highlight the difficulties to deal with such an equation both analytically and numerically. Some recent theoretical progress has been built nevertheless regarding this type of boundary value problems as well some generalizations~\cite{BE2016,EGHPT2015,EGP2015,EP2013,ET2015,ET2017}.
The reader may appreciate in these references the wide range of mathematical techniques needed to analyze this sort of differential equations.

The aim of the present work is to find the numerically approximated solutions of the fourth order differential equation (\ref{Eq 3}) with $G(r) \equiv 1$. To get the solutions of (\ref{Eq 3}) we first compute the solutions of differential equation (\ref{Eq 5}) using an iterative numerical scheme and compare our results to the ones in \cite{CarlosRadial2013}. This numerical method became popularized in recent years under the name of variational iteration method (VIM) but it is in
fact a reformulation of classical schemes (see \cite{Anderson1980, Ramos2008, HE1999, HE2007, HEXHWU2007, VVA1981, ASKA2010, TES1933, LAAS2010, SAAS2012}). Recently,   VIM are still under investigation, e.g., Wazwaz et al. (\cite{wazwaz2011}) used it to find the approximate solution of nonlinear singular boundary value problem, Zhang et al. (\cite{XZFYAM2018}) applied it on a family of fifth-order convergent methods for solving nonlinear equations, Zellal et al. (\cite{MKB2018}) used it on biological population model, Singh et al. (\cite{MARP2019}) discussed  it on a 2 point and 3 point nonlinear SBVPs.

The remainder of the paper has been organized as follows. In section~\ref{nm}, we describe the numerical method, show that it is well suited to approach the
present boundary value problems, and illustrate how it works by solving explicitly a linearization of our nonlinear differential equation. In section~\ref{numerics}, we use this method to solve numerically the non self-adjoint nonlinear singular boundary value problems under study and show a wide range of numerical results. In section \ref{P1Tables} and \ref{P1Figures}, we place the numerical data.
Finally, in section~\ref{conclusions} we draw our main conclusions.

\section{The numerical method}\label{nm}

In order to explain the  phenomenology of this method we consider a general non-linear differential equation of the form
\begin{equation}
 \begin{aligned}\label{Eq 6}
& Lw(r)+Nw(r)=f(r),&
 \end{aligned}
\end{equation}
where $L$ is the linear operator, $N$ is the nonlinear operator and $f(r)$  is a known function. We can construct the  correction functional  to our iterative
scheme applied to equation (\ref{Eq 6})   as follows:
\begin{equation}
\begin{aligned}\label{Eq 7}
&  \displaystyle w_{n+1}(r)=w_{n}(r)+\int_{0}^{r} \mu (t)\left(Lw_{n}(t)+N\tilde{w}_{n}(t)-f(t)\right) dt,
 \end{aligned}
\end{equation}
where $\mu(t) $  is the Lagrange   multiplier and $\tilde{w}_{n}$ is of restricted variation i.e., $\delta \tilde{w}_{n}=0$ (\cite{AKV2015}). The Lagrange multiplier $\mu(t) $ can be identified optimally via the variational principle  (\cite{Variationalprinciple}) and integration by parts. After getting the value of $\mu(t) $, we arrive at a recurrence relation defined by equation (\ref{Eq 7}). We take a suitable initial approximation $w_{0}(r)$ in such a way such that this integral in equation (\ref{Eq 7}) is convergent and hence we can compute $w_{1}(r)$, $w_{2}(r),\cdots$. The exact solution $w(r) $ of equation (\ref{Eq 6}) can be found as
\begin{equation}
\begin{aligned}\label{Eq 8}
&w(r)=\lim_{n\rightarrow \infty} w_{n}(r).
\end{aligned}
\end{equation}
Before we solve equation (\ref{Eq 5}) using (\ref{Eq 7}) we shall discuss the properties of  this method for the nonlinear singular boundary value problems associated to (\ref{Eq 5}). From equation (\ref{Eq 7}) and equation (\ref{Eq 5}) we get the correction functional as
\begin{equation}
\begin{aligned}\label{Eq 9}
&\displaystyle w_{n+1}(r)=w_{n}(r)+\int_{0}^{r} \mu(t)\left(t^{2}w_{n}''(t)-t \tilde{w}_{n}'(t)-\frac{1}{2}\tilde{w}_{n}^{2}(t)-\frac{1}{2}\lambda t^{4}\right) dt.
\end{aligned}
\end{equation}
To find the value of $\mu(t)$ we take variations on both sides of (\ref{Eq 9}) such that $\delta \left(t \tilde{w}_{n}'(t)+\frac{1}{2}\tilde{w}_{n}^{2}(t)\right)=0$. Eq (\ref{Eq 9}) becomes
\begin{equation}
\begin{aligned}\label{Eq 10}
\displaystyle \delta w_{n+1}(r)=\delta w_{n}(r)+\delta \int_{0}^{r} \mu(t)\left(t^{2}w_{n}''(t)-t \tilde{w}_{n}'(t)-\frac{1}{2}\tilde{w}_{n}^{2}(t)-\frac{1}{2}\lambda t^{4}\right) dt,
\end{aligned}
\end{equation}
or equivalently,
\begin{equation}
\begin{aligned}\label{Eq 11}
&\displaystyle \delta w_{n+1}(r)=\delta w_{n}(r)+\delta \int_{0}^{r} \mu(t)t^{2}w_{n}''(t) dt.
\end{aligned}
\end{equation}
By using integration by parts we get
\begin{equation}
\begin{aligned}\label{Eq 12}
\displaystyle &\delta w_{n+1}(r)=\left( 1-\mu'(r) r^{2}-2 r \mu(r) \right) \delta w_{n}(r)+\mu(r) r^{2}\delta w_{n}'(r)+\int_{0}^{r} \left(\mu''(t) t^{2}+4 t \mu'(t)+2 \mu(t) \right) \delta w_{n} dt.
 \end{aligned}
\end{equation}
The extremum condition imposed on $ w_{n+1}$ gives $\delta w_{n+1}=0$ (\cite{AKV2015}), and from (\ref{Eq 12}) we arrive at the following stationary conditions
\begin{equation}
\begin{aligned}\label{Eq 13}
\hspace{-.4cm}
&1-\mu'(r) r^{2}-2 r \mu(r)=0,
\end{aligned}
\end{equation}
\begin{equation}
\begin{aligned}\label{Eq 14}
\hspace{-2.3cm}
&\mu(r)=0,
\end{aligned}
\end{equation}
\begin{equation}
\begin{aligned}\label{Eq 15}
& \mu''(t) t^{2}+4 t \mu'(t)+2 \mu(t)=0.
\end{aligned}
\end{equation}
By solving (\ref{Eq 13}-\ref{Eq 15}), we get the optimal value of $\mu(t)$, given by
\begin{equation}
\begin{aligned}\label{Eq 16}
\mu(t)=\frac{t-r}{t^{2}}.
\end{aligned}
\end{equation}
Hence the correction functional (\ref{Eq 9}) becomes
\begin{equation}
\begin{aligned}\label{Eq 17}
&\displaystyle w_{n+1}(r)=w_{n}(r)+\int_{0}^{r} \frac{t-r}{t^{2}}\left(t^{2}w_{n}''(t)-t {w}_{n}'(t)-\frac{1}{2}{w}_{n}^{2}(t)-\frac{1}{2}\lambda t^{4}\right) dt.
\end{aligned}
\end{equation}
Equation (\ref{Eq 17}) gives rise to a sequence \{$w_{n}(r)$\}. If this sequence \{$w_{n}(r)$\} is convergent then its limit gives the exact solution of the nonlinear singular differential equation (\ref{Eq 5}). Now we shall prove that this sequence \{$w_{n}(r)$\} is well defined.
This fact is established by means of the following Lemma.

\begin{lemma}\label{corollary1}
Let $w_{n}(r),~n=0,1,2,\cdots$,  satisfy the following properties:
\begin{itemize}
\item $P_{1}: ~w_{n}'(0)=0,$
\end{itemize}
\begin{itemize}
\item $\displaystyle P_{2}:~ \lim_{r\rightarrow 0^{+}}w_{n}(r)=0,$
\end{itemize}
\begin{itemize}
\item $P_{3}:~ w_{n}''(r) \in C^2([0,1]).$
\end{itemize}
Then $\displaystyle \frac{w_{n}(r)}{r^{2}}$ and
$\displaystyle \frac{w_{n}'(r)}{r}$ are bounded on $[0,1]$ for all $n\in \mathbb{N}_0$.
\end{lemma}

\begin{proof} We have $\displaystyle \lim_{r \rightarrow 0^{+}} \frac{w_{0}(r)}{r^{2}}=\displaystyle \lim_{r \rightarrow 0^{+}} \frac{w_{0}'(r)}{2r}=\displaystyle \lim_{r \rightarrow 0^{+}} \frac{w_{0}''(r)}{2} $. But $w_{0}''(r)$ is bounded on $[0,1]$. So we conclude that $\displaystyle \frac{w_{0}(r)}{r^{2}}$ is bounded on $[0,1]$. By a similar argument, since $P_{1}$, $P_{2}$, and $P_{3}$ hold true for all $n\in \mathbb{N}$, $\displaystyle \frac{w_{n}(r)}{r^{2}}$ is bounded on $[0,1]$ for all $n\in \mathbb{N}$. The boundedness of $\displaystyle \frac{w_{n}'(r)}{r}$ follows from
the same argument.
\end{proof}

\begin{lemma}\label{Lemma1}
Let $w_{0}(r)$ be the initial approximation of (\ref{Eq 17}) such that
\begin{equation}\label{Eq 44}
 w_{0}'(0)=0, \displaystyle\lim_{r\rightarrow 0^{+}}w_{0}(r)=0~and~ w_{0}''(r) \in C^2([0,1]).
\end{equation} Let $w_{n}(r),~n=1,2,\cdots$, be  defined by (\ref{Eq 17}). Then $w_{n}(r),~n=1,2,\cdots$,  satisfies the following properties:
\begin{itemize}\label{item1}
\item $P_{1}: ~w_{n}'(0)=0,$
\end{itemize}
\begin{itemize}
\item $\displaystyle P_{2}:~ \lim_{r\rightarrow 0^{+}}w_{n}(r)=0,$
\end{itemize}
\begin{itemize}
\item $P_{3}:~ w_{n}''(r) \in C^2([0,1]).$
\end{itemize}
\end{lemma}

\begin{proof}
We proof the statement by induction. For  $n=0$ properties $P_{1}$, $P_{2}$, and $P_{3}$ are true. Now we prove that $P_{1}$, $P_{2}$, and $P_{3}$ are true for $n=1$. From equation (\ref{Eq 17}) by setting $n=0$ we get
\begin{equation}\label{Eq 45}
\displaystyle w_{1}(r)=w_{0}(r)+\int_{0}^{r} \frac{t-r}{t^{2}}\left(t^{2}w_{0}''(t)-t {w}_{0}'(t)-\frac{1}{2}{w}_{0}^{2}(t)-\frac{1}{2}\lambda t^{4}\right) dt.
\end{equation}
By differentiating both sides of equation (\ref{Eq 45}) with respect to $r$ we get
\begin{eqnarray}\label{Eq  46}
\nonumber\displaystyle w_{1}'(r)=w_{0}'(r)+\frac{r-r}{r^{2}}\left(r^{2}w_{0}''(r)-r {w}_{0}'(r)-\frac{1}{2}{w}_{0}^{2}(r)-\frac{1}{2}\lambda r^{4}\right)\frac{d}{dr}(r)\\
\nonumber-\lim_{t\rightarrow 0^{+}} \left( \frac{t-r}{t^{2}}\left(t^{2}w_{0}''(t)-t {w}_{0}'(t)-\frac{1}{2}{w}_{0}^{2}(t)-\frac{1}{2}\lambda t^{4}\right)\right)\frac{d}{dr}(0) \\
+\int_{0}^{r} \frac{-1}{t^{2}}\left(t^{2}w_{0}''(t)-t {w}_{0}'(t)-\frac{1}{2}{w}_{0}^{2}(t)-\frac{1}{2}\lambda t^{4}\right) dt.
\end{eqnarray}
Using (\ref{Eq 44}) we can easily get that
\begin{equation}\label{Eq 73}
w_{1}'(r)=w_{0}'(r)+\int_{0}^{r} \frac{-1}{t^{2}}\left(t^{2}w_{0}''(t)-t {w}_{0}'(t)-\frac{1}{2}{w}_{0}^{2}(t)-\frac{1}{2}\lambda t^{4}\right) dt
\end{equation}
and
\begin{equation}\label{Eq 47}
\displaystyle w_{1}''(r)=w_{0}''(r)- \frac{1}{r^{2}}\left(r^{2}w_{0}''(r)-r {w}_{0}'(r)-\frac{1}{2}{w}_{0}^{2}(r)-\frac{1}{2}\lambda r^{4}\right).
\end{equation}
By setting $r=0$ in equation (\ref{Eq 73}) we get $\displaystyle w_{1}'(0)=w_{0}'(0)=0$ via the application of Lemma~\ref{corollary1}.

Using equation (\ref{Eq 44}) and Lemma~\ref{corollary1} it can be easily concluded that the integrand in the integral of (\ref{Eq 45}) is bounded on $[0,1]$. Taking $\displaystyle \lim_{r\rightarrow 0^{+}}$ on both sides of equation (\ref{Eq 45}), we get $\displaystyle \lim_{r\rightarrow 0^{+}} w_{1}(r)=\lim_{r\rightarrow 0^{+}} w_{0}(r)+0=0.$

Now from equation (\ref{Eq 47}), we get
\begin{equation}\label{Eq 49}
\left|w_{1}''(r)\right|\leq \left|w_{0}''(r)\right|+ \left|\frac{1}{r^{2}}\left(r^{2}w_{0}''(r)-r {w}_{0}'(r)-\frac{1}{2}{w}_{0}^{2}(r)-\frac{1}{2}\lambda r^{4}\right)\right|.
\end{equation}
A simple application of \eqref{Eq 44} and Lemma~\ref{corollary1} on the second term of the right hand side of equation (\ref{Eq 49}) reveals that it is bounded on $[0,1]$. Thus $w_{1}''(r)$ is bounded on $[0,1].$ Moreover $w_{1}''(r)$ is continuous with the obvious definition of $w_{1}''(0)$, as can be immediately checked from equation (\ref{Eq 47}) and an argument akin to the one in the proof of Lemma~\ref{corollary1}.

Therefore $P_{1}$, $P_{2}$, and $P_{3}$ hold for $n=1$.

Let us now assume that $P_{1}$, $P_{2}$, and $P_{3}$ hold for an arbitrary $n$, that is
\begin{equation}\label{Eq 50}
w_{n}'(0)=0,~~\lim_{r \rightarrow 0^{+}} w_{n}(r)=0~~\text{and}~~ w _{n}''(r)
\end{equation}
is  bounded and continuous on $[0,1]$.

Now we will show that $P_{1}$, $P_{2}$, and $P_{3}$ hold for $n+1$ too. From equation (\ref{Eq 17})  we get
\begin{equation}\label{Eq 51}
\displaystyle w_{n+1}(r)=w_{n}(r)+\int_{0}^{r} \frac{t-r}{t^{2}}\left(t^{2}w_{n}''(t)-t {w}_{n}'(t)-\frac{1}{2}{w}_{n}^{2}(t)-\frac{1}{2}\lambda t^{4}\right) dt.
\end{equation}
Differentiating both sides  with respect to $r$, we get
\begin{equation}\label{Eq 52}
\displaystyle w_{n+1}'(r)=w_{n}'(r)+\int_{0}^{r} \frac{-1}{t^{2}}\left(t^{2}w_{n}''(t)-t {w}_{n}'(t)-\frac{1}{2}{w}_{n}^{2}(t)-\frac{1}{2}\lambda t^{4}\right) dt,
\end{equation}
and
\begin{equation}\label{Eq 53}
\displaystyle w_{n+1}''(r)=w_{n}''(r)- \frac{1}{r^{2}}\left(r^{2}w_{n}''(r)-r {w}_{n}'(r)-\frac{1}{2}{w}_{n}^{2}(r)-\frac{1}{2}\lambda r^{4}\right).
\end{equation}
By setting $r=0$ in equation (\ref{Eq 52}) we get $\displaystyle w_{n+1}'(0)=w_{n}'(0)=0$ by Lemma~\ref{corollary1}.

Using equation (\ref{Eq 50}) and Lemma~\ref{corollary1} it is easy to see that the integrand inside the integral of (\ref{Eq 51}) is bounded on $[0,1]$. Taking $\displaystyle \lim_{r\rightarrow 0^{+}}$ on both sides of equation (\ref{Eq 51}) we get $\displaystyle \lim_{r\rightarrow 0^{+}} w_{n+1}(r)=\lim_{r\rightarrow 0^{+}} w_{n}(r)+0=0.$

Now from equation (\ref{Eq 53}) we get
\begin{equation}\label{Eq 55}
\left|w_{n+1}''(r)\right|\leq \left|w_{n}''(r)\right|+ \left|\frac{1}{r^{2}}\left(r^{2}w_{n}''(r)-r {w}_{n}'(r)-\frac{1}{2}{w}_{n}^{2}(r)-\frac{1}{2}\lambda r^{4}\right)\right|.
 \end{equation}
A simple application of \eqref{Eq 50} and Lemma~\ref{corollary1} on the second term of the right hand side of equation (\ref{Eq 55}) shows that it is bounded on $[0,1]$. Thus $w_{n+1}''(r)$ is bounded on $[0,1]$ as well. Furthermore $w_{n}''(r)$ is continuous with the evident definition of $w_{n}''(0)$, as can be seen from equation (\ref{Eq 53}) and an analogous argument to that in the proof of Lemma~\ref{corollary1}.

Thus $P_{1}$, $P_{2}$, and $P_{3}$ hold true for $n+1$.

By mathematical induction we conclude that $P_{1}$, $P_{2}$, and $P_{3}$ hold true for all $n\in \mathbb{N}$.
\end{proof}

Lemmata~\ref{corollary1} and~\ref{Lemma1} show that the iterations in our numerical method lead to a well-defined procedure;
now we will illustrate how the method works with a particular example. If $\lambda = 0$ then $w(r) \equiv 0$ is obviously a solution.
And if $|\lambda|$ is very small then the nonlinearity in equation~\eqref{Eq 5} is $O(\lambda^2)$ and therefore negligible.
In this limit the equation to be solved is
\begin{equation}\label{linear}
r^{2} w''(r)-r w'(r)=\frac{1}{2} \lambda r^{4}.
\end{equation}
This is a linear differential equation of Euler type and therefore explicitly solvable, but instead using standard techniques we will proceed to
solve it using our numerical method. A homogeneity argument suggests using an arbitrary fourth degree polynomial as initial condition for the iterations, that is
$$
w_0(r)= a_4 r^4 + a_3 r^3 + a_2 r^2 + a_1 r + a_0,
$$
where $a_4, a_3, a_2, a_1, a_0 \in \mathbb{R}$.
The first iteration shows that the only way to maintain this procedure finite is to set $a_1=0$, so we will do so from now on.
Now it is easy to compute the $n^{\text{th}}$ iteration to find
$$
w_n(r)= \left( \frac{a_4}{3^n} + \frac{\lambda}{16} - \frac{\lambda}{16 \times 3^n} \right) r^4
+ \frac{a_3}{2^n} \, r^3 + a_2 r^2 + a_0.
$$
Also, it is immediate to check that the limit
$$
w_\infty(r) := \lim_{n \to \infty} w_n(r)= \frac{\lambda}{16} r^4 + a_2 r^2 + a_0
$$
solves the equation~\eqref{linear} for any $a_2, a_0 \in \mathbb{R}$. These parameters are fixed by the boundary conditions, in particular,
the condition $\phi'(0)=0$, common to both boundary value problems, translates to
$$
\lim_{r \to 0^+} \frac{w(r)}{r}=0 \Longrightarrow a_0=0.
$$
Finally, for the Dirichlet problem $w(1)=0$ and we find
$$
w_D(r) = \frac{\lambda}{16} \, r^2 (r^2 -1),
$$
 for the Navier problem of type one $w'(1)=0$ and we get
$$
w_{N_{1}}(r) = \frac{\lambda}{16} \, r^2 (r^2 -2),
$$
and for the Navier problem of type two $w(1)=w'(1)$ and we get
$$
w_{N_{2}}(r) = \frac{\lambda}{16} \, r^2 (r^2 -3).
$$
These results together with the common boundary condition $\phi(1)=0$ yield the linear approximations
\begin{eqnarray}\nonumber
\phi_D(r) &=& \frac{\lambda}{64} (r^2 -1)^2, \\
 \nonumber \phi_{N_{1}}(r) &=& \frac{\lambda}{64} (r^4 -4 r^2 +3),\\
\nonumber \phi_{N_{2}}(r) &=& \frac{\lambda}{64} (r^4 -6 r^2 +5),
\end{eqnarray}
to the solutions of equation~\eqref{Eq 3} with $G(r) \equiv 1$, i.~e. solutions to the linear equation
$$
\frac{1}{r} \left \{ r\left[\frac{1}{r}(r \phi')^{'}\right]'\right \}^{'} = \lambda,
$$
for the respective sets of boundary conditions. These linear approximations are represented in Figure~\ref{linearfig}.

A qualitative agreement can be appreciated between the linear approximations and the actual solutions to the nonlinear
problems numerically computed in the next section. Of course, multiplicity of solutions is not found in the linear regime, this phenomenon
only appears when the full nonlinear problem is considered, see section~\ref{numerics}. In this respect, these linear approximations approximate the solutions of the nonlinear problem that lie in vicinity of the origin, the so called lower solutions in the next section. One can appreciate the common geometric features shared by nonlinear lower solutions and linear approximations. Another phenomenon that does not appear at the linear approximation level is non-existence of solutions: this takes place for values of $\lambda$ that lie beyond the range of validity of this approximation. Both phenomena of multiplicity and non-existence of solutions to the nonlinear boundary value problems are discussed in detail in the following section.

All of these considerations show how our numerical method works, and we will now use it to numerically solve the nonlinear problems
under study in the next section.

\begin{figure}[H]
\centering
\subfigure[\,\,$\lambda=1/2$]{\includegraphics[width=.45\linewidth]{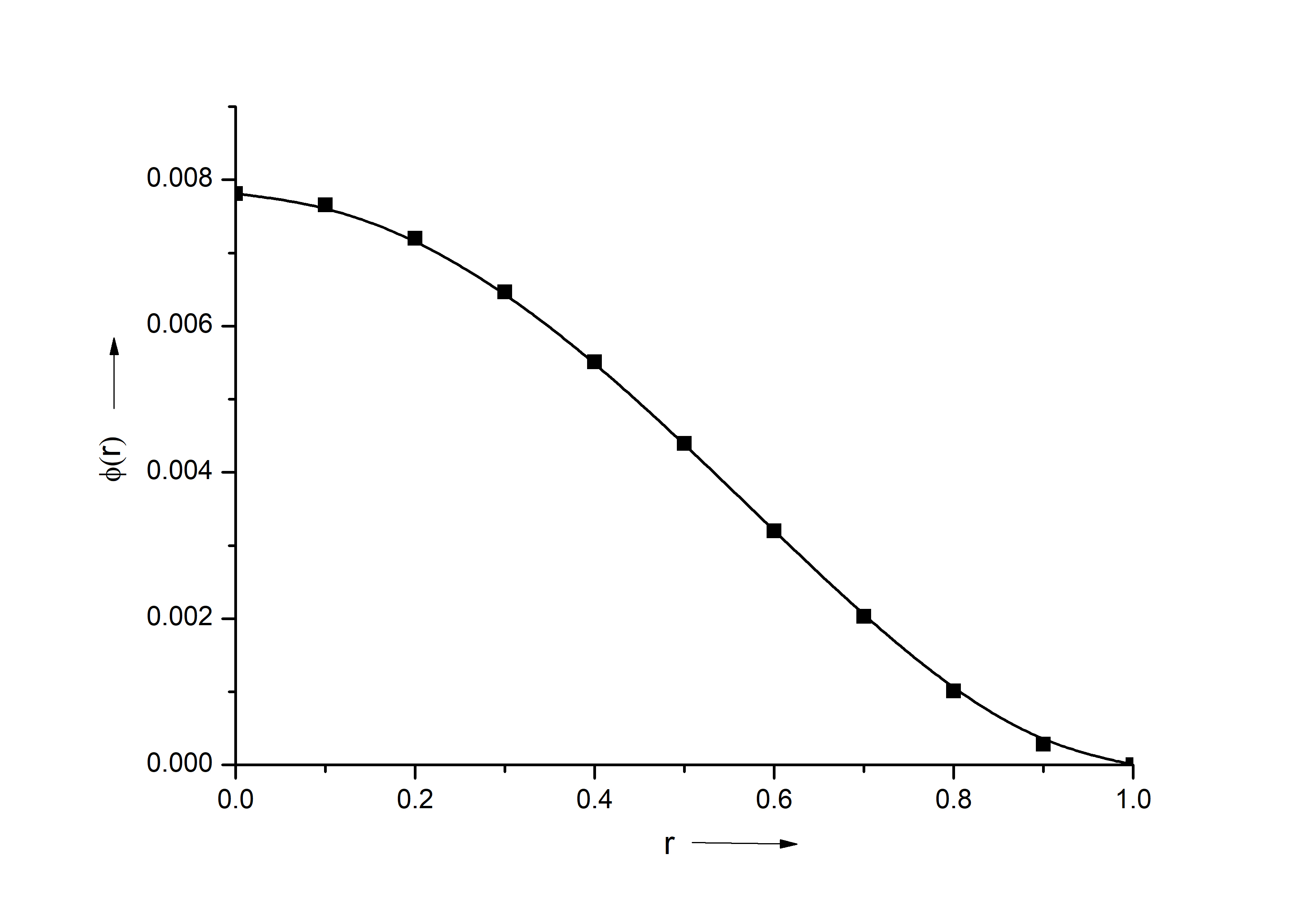}} \hspace{.4cm}
\subfigure[\,\,$\lambda=1$]{\includegraphics[width=.45\linewidth]{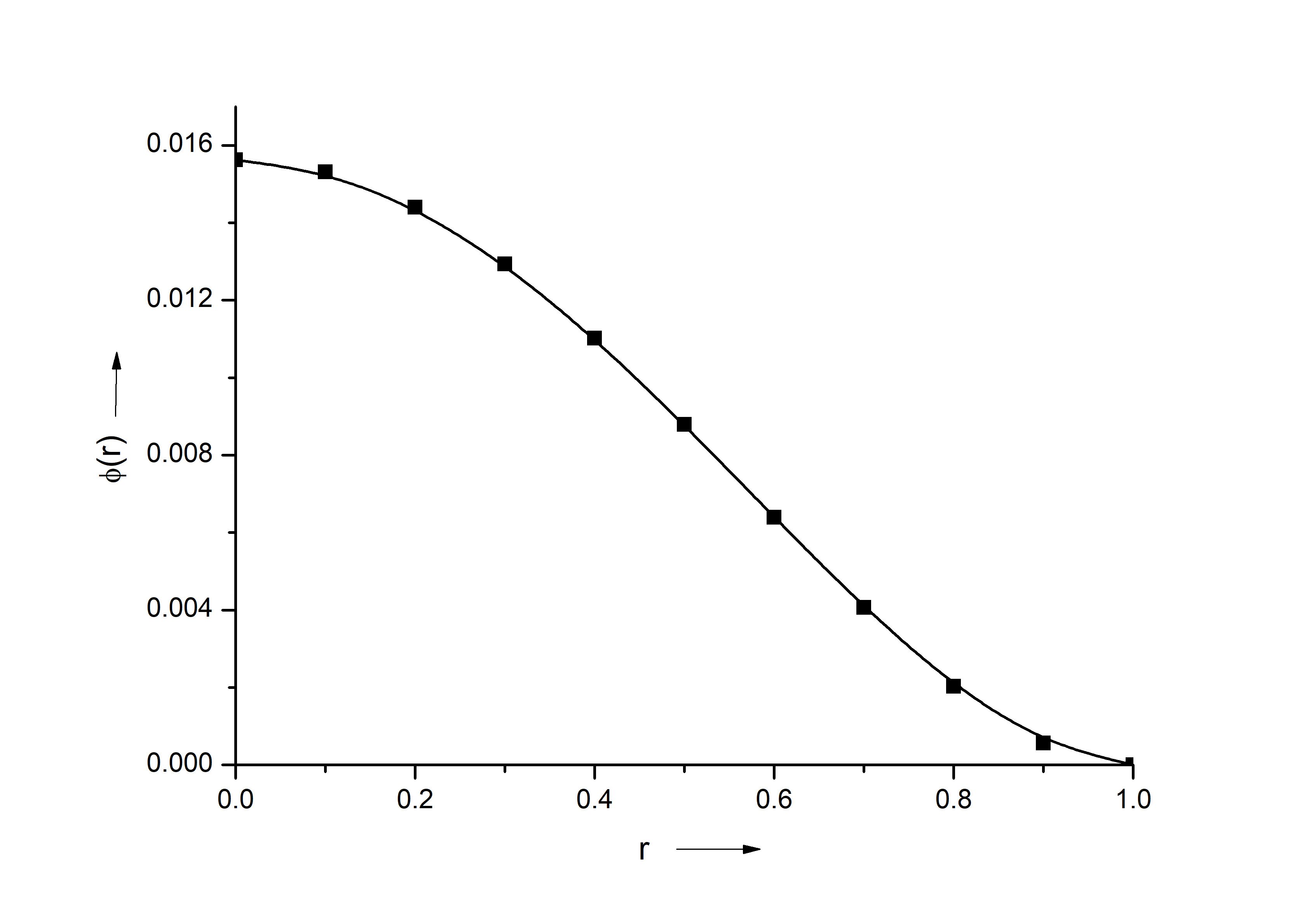}}
\end{figure}
\begin{figure}[H]
\centering
\subfigure[\,\,$\lambda=1/2$]{\includegraphics[width=.45\linewidth]{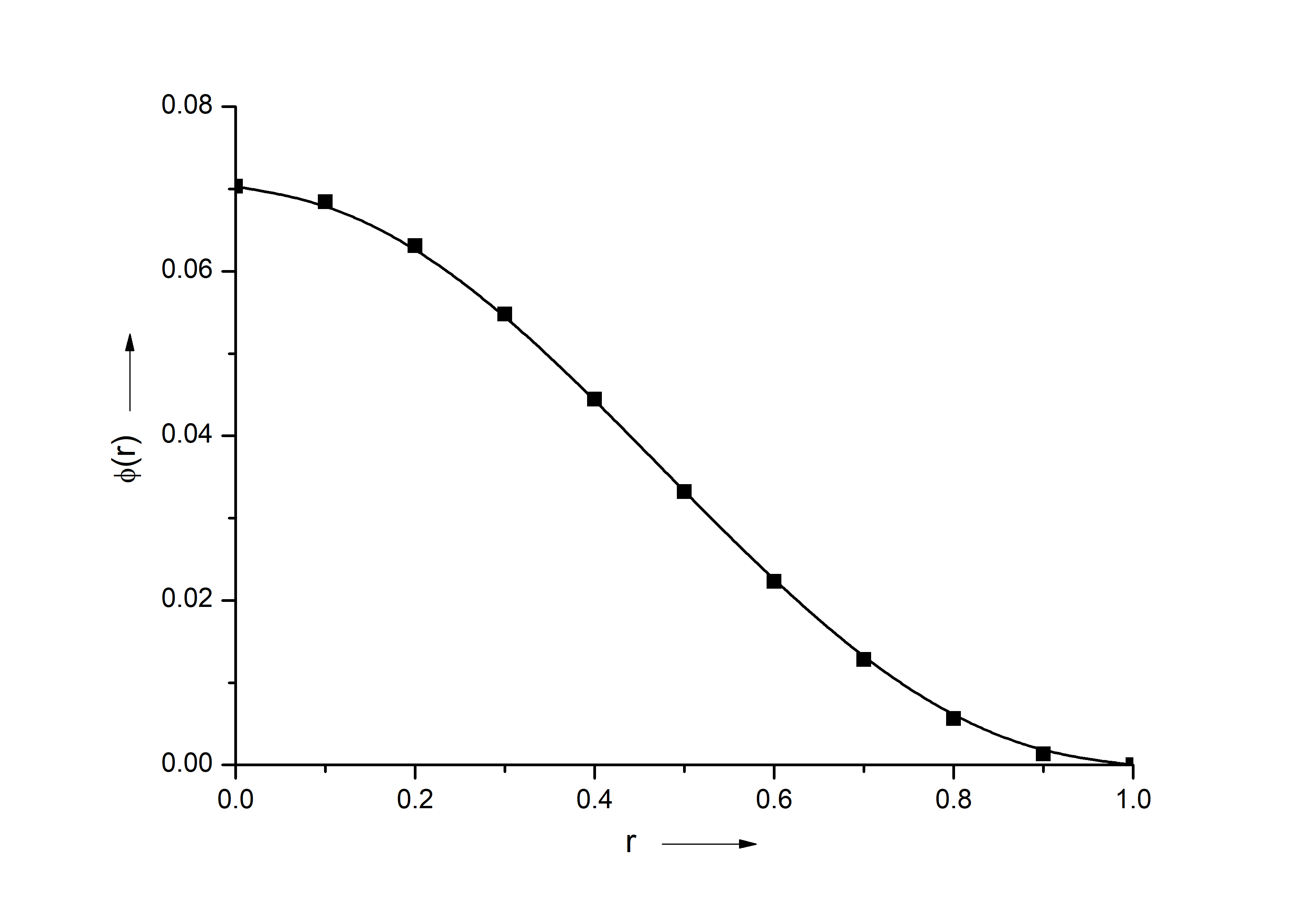}} \hspace{.4cm}
\subfigure[\,\,$\lambda=1$]{\includegraphics[width=.45\linewidth]{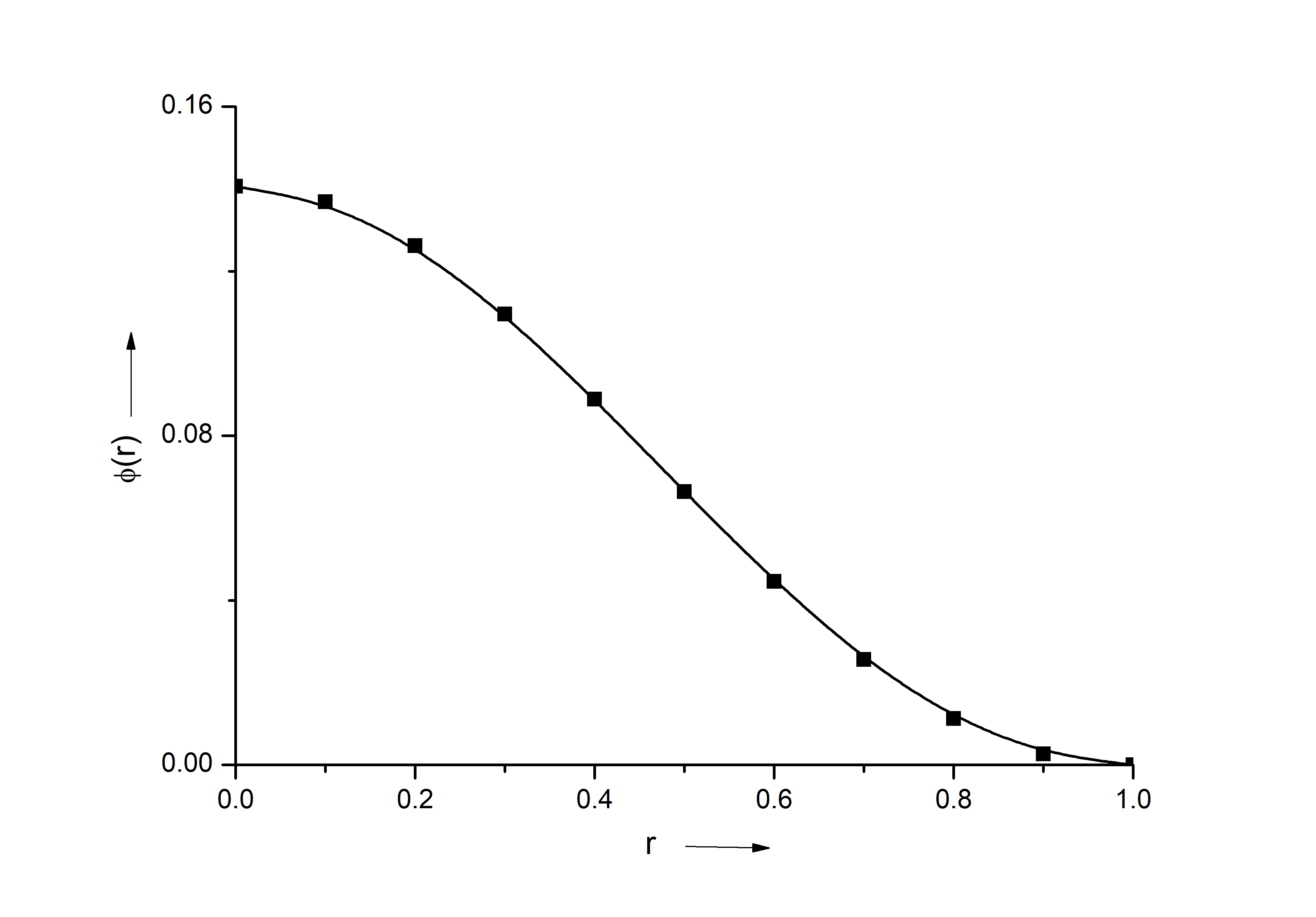}}
\end{figure}
\begin{figure}[H]
\centering
\subfigure[\,\,$\lambda=1/2$]{\includegraphics[width=.45\linewidth]{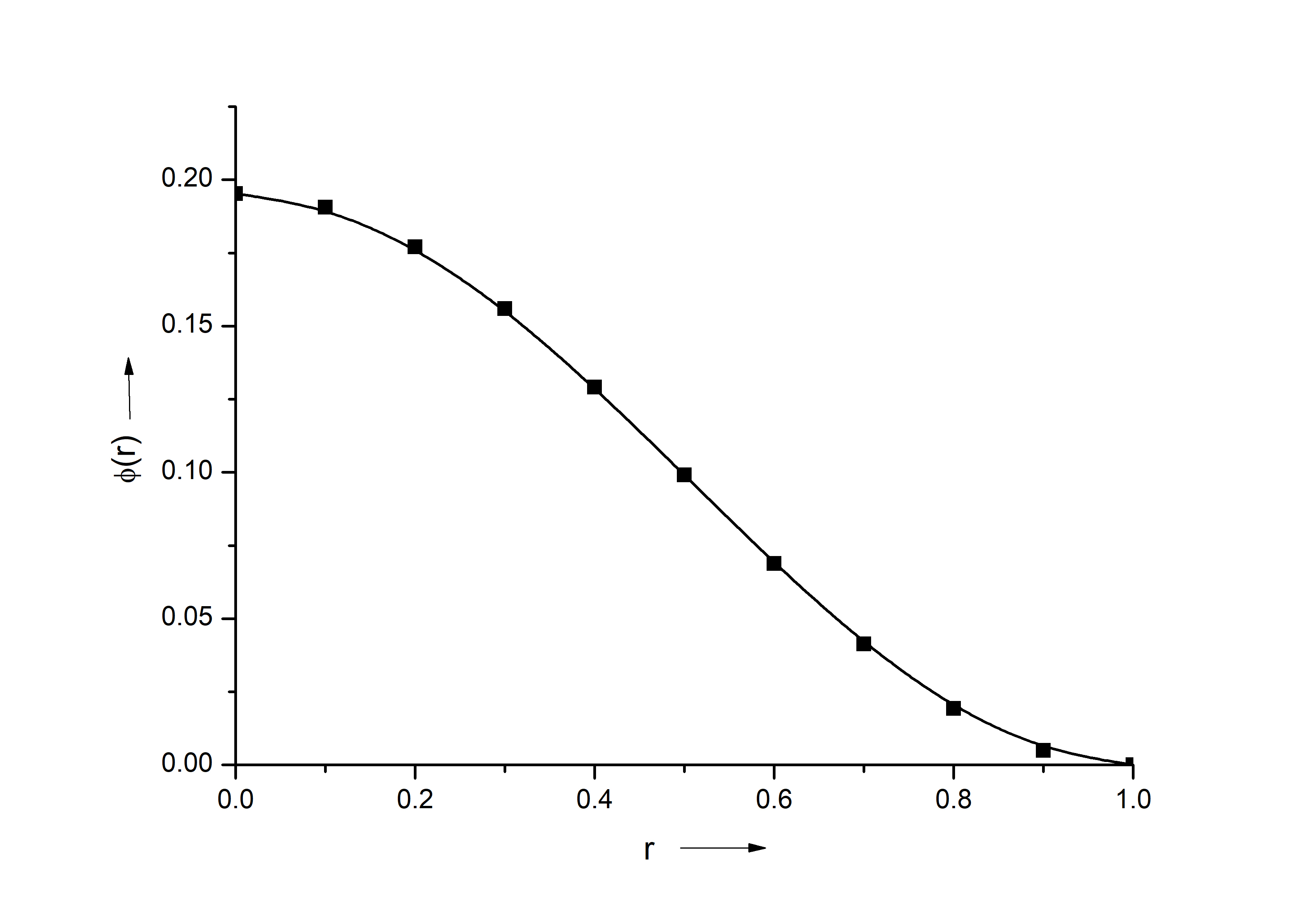}} \hspace{.4cm}
\subfigure[\,\,$\lambda=1$]{\includegraphics[width=.45\linewidth]{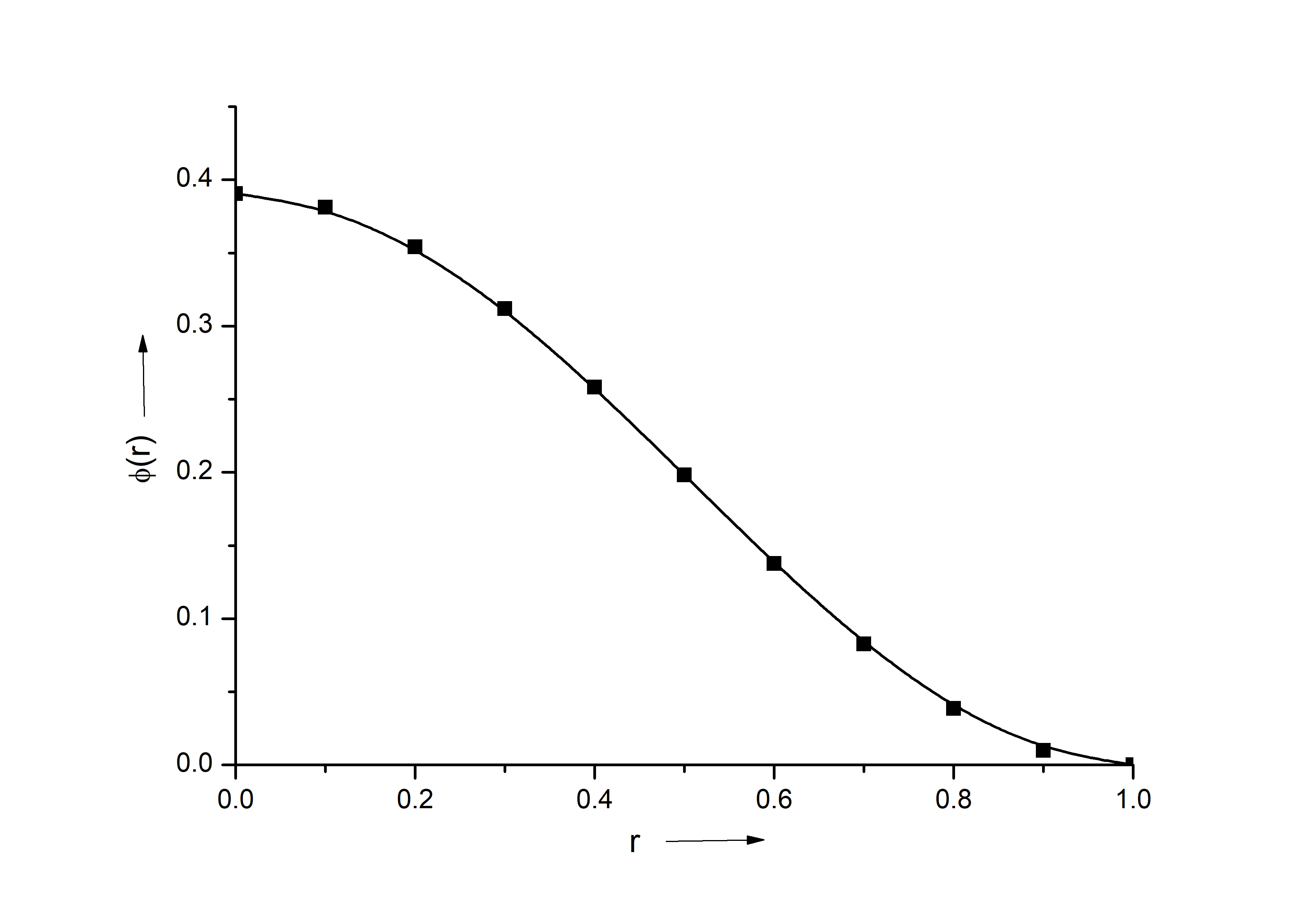}}
\caption{Linear approximations to the solutions of equation~\eqref{Eq 3} with $G(r) \equiv 1$
for Dirichlet boundary conditions with $\lambda=1/2$ (figure (a)) and $\lambda=1$ (figure (b)),
and for Navier boundary conditions of type one with $\lambda=1/2$ (figure (c))  $\lambda=1$ (figure (d)), and for Navier boundary conditions of type two with $\lambda=1/2$ (figure (e)) and $\lambda=1$ (figure (f)),
as explained in the text. Of course, only the solution in the vicinity of the origin is approximated.}
\label{linearfig}
\end{figure}

\section{Numerical illustrations}\label{numerics}

In this section we will discuss about the solvability of differential equation (\ref{Eq 3}) associated to either homogeneous Dirichlet boundary conditions or homogeneous Navier boundary conditions of both types. For this purpose we first apply our iterative numerical scheme on equation (\ref{Eq 5}) and after that using the transformation $w(r)=r \phi'(r)$ and the left boundary conditions we get the solution of (\ref{Eq 3}). In \cite{Carlos2012}, they arrive at two cases:

\noindent{\em \textbf{Case (a)}: $\lambda \geq 0.$}\label{P1Casea}\\
They observe that for $\lambda=0$ there are two solutions. One is trivial and the other is non trivial. For $0<\lambda < \lambda_{\text{critical}}$ they always get two non-trivial solutions. Since the solutions are ordered, they call them respectively the upper and lower solution. The critical value of $\lambda$, i.e. $\lambda_{\text{critical}}$, is approximated to be $11.34$ and $169$ corresponding to Navier boundary condition of type two and Dirichlet boundary condition (\cite{Carlos2012}) respectively. For $\lambda> \lambda_{\text{critical}}$ there does not exist any numerical solutions. No conclusion is given for Navier boundary condition of type one.

\noindent{\em \textbf{Case (b)}: $\lambda < 0.$}\\
No conclusion is given for both types of boundary conditions.
\subsection{Navier boundary condition of type one}\label{P1typetwo}
Here we consider equation (\ref{Eq 5}) subject to
\begin{eqnarray}
w'(0)=0 ~\mbox{and}~w'(1)=0.
\end{eqnarray}
Let $ w_{0}(r)= a r^{2},~~a\in \mathbb{R},$ be an initial approximation. This choice satisfies all the assumptions of Lemma \ref{Lemma1}. By  setting  $n=0$ in equation (\ref{Eq 17}) we get
\begin{equation}\label{Eq 39}
\displaystyle w_{1}(r)=w_{0}(r)+\int_{0}^{r} \frac{t-r}{t^{2}}\left(t^{2}w_{0}''(t)-t {w}_{0}'(t)-\frac{1}{2}{w}_{0}^{2}(t)-\frac{1}{2}\lambda t^{4}\right) dt.
\end{equation}
Successively we get
\begin{eqnarray}
&&\label{Eq 40}
w_1(r)=\frac{a^2 r^4}{24}+a r^2+\frac{\lambda  r^4}{24},\\
&&\label{Eq 41}
w_2(r)=\frac{a^4 r^8}{64512}+\frac{a^3 r^6}{720}+\frac{a^2 \lambda  r^8}{32256}+\frac{a^2 r^4}{18}+\frac{1}{720} a \lambda  r^6+a r^2+\frac{\lambda ^2 r^8}{64512}+\frac{\lambda  r^4}{18},
\end{eqnarray}
and so on.

By using Mathematica we have computed $w_n$ for $n=3,4,5,6,7$, but due to lack of space we could not list all of them. To calculate $\phi(r)$, here  we take $\displaystyle w(r)=w_{7}(r)$ as an approximate solution of (\ref{Eq 5}). So, we found  $w(r)$ is a function of $r$ including  $a$ as a constant and  $\lambda$ as a parameter. By using the boundary condition $w'(1)=0$ we have computed the values of $a$ corresponding to each $\lambda$. Now from transformation $ w(r)=r \phi^{'}(r) $ and boundary condition $\phi(1)=0$ we have computed the solution $\phi(r)$ corresponding to different values of $\lambda$. We arrive at two cases depending on $\lambda$.

\noindent{\em \textbf{Case (c)}: $\lambda \geq 0.$}\\
Here, We observe the same remarks as in Case $(a)$. The critical value of $\lambda$, i.e. $\lambda_{\text{critical}}$, is approximately $31.94$. For $\lambda> \lambda_{\text{critical}}$ numerical solutions do not exist as the values of $a$ become imaginary. In table \ref {P1Table1} and table \ref{P1Table2} we tabulate residual errors.  We have also plotted the graphs of $\phi(r)$. We see that the solutions  are moving towards each other as $\lambda$ is increased [see Figure \ref{P1Figure1}].

\noindent{\em \textbf{Case (d):} $\lambda < 0$.}\\
We observe that we always get two numerical solutions corresponding to each negative $\lambda$. One solution is positive (namely the positive solution) and the other solution is negative (namely the negative solution). There is no negative critical $\lambda$. We list residue errors in tables \ref{P1Table3} and \ref{P1Table4}. We see that the solutions  are aparting from each other as $\lambda$ is decreased [see Figure \ref{P1Figure2}].
\subsection{Navier boundary condition of type two}

In this subsection, we consider :
\begin{eqnarray}
w'(0)=0~\mbox{and}~w(1)=w'(1)
\end{eqnarray}
corresponding to equation (\ref{Eq 5}).

Here we also have the same iterations as in (\ref{Eq 40}) and (\ref{Eq 41}).  To calculate $\phi(r)$, here  we take $\displaystyle w(r)=w_{7}(r)$ as an approximate solution of (\ref{Eq 5}). By using the boundary condition $w(1)=w'(1)$,  $ w(r)=r \phi^{'}(r) $ and boundary condition $\phi(1)=0$, we have computed the solution $\phi(r)$ corresponding to different values of $\lambda$. We arrive at two cases depending on $\lambda$.

\noindent{\em \textbf{Case (e)}: $\lambda \geq 0.$}\\
Same remarks are made as in Case (c) to the subsection \ref{P1typetwo}. The critical value of $\lambda$, i.e. $\lambda_{\text{critical}}$, is approximately $11.34$ (\cite{Carlos2012}).  In table \ref {P1Table5} and table \ref{P1Table6} we tabulate residual errors, and in figure \ref{P1Figure3} we place the approximate solutions graph.

\noindent{\em \textbf{Case (f):} $\lambda < 0.$}\\
All  remarks are same to Case (d) in subsection \ref{P1typetwo}. In table \ref {P1Table7}, table \ref{P1Table8}  and figure \ref{P1Figure4}  we place the numerical data of  approximate solutions.
\subsection{Dirichlet boundary condition}
Furthermore. here we consider Dirichlet boundary conditions 
\begin{eqnarray}
 w'(0)=0 ~\mbox{and}~\quad w(1)=0.
\end{eqnarray}
The iteration scheme is given by equation (\ref{Eq 17}).

Here  we also consider $w_0(r)=a r^{2},~~a\in \mathbb{R},$ as an initial approximation. Then we compute $w_1,w_2,w_3,w_4,w_5,w_6$. We take $w_{6}(r)$ as approximate solution. To calculate the values of $a$ we use the boundary condition $w(1)=0$. Then using the transformation $w(r)=r \phi^{'}$ and $ \phi(1)=0$ we easily get the approximate solution $\phi(r)$. Here we also arrive at two different cases depending on the value of parameter $\lambda$.

\noindent{\em \textbf{Case (g):} $\lambda \geq 0$.}\\
Here  we  have also noticed analogous remarks as in subsection \ref{P1typetwo}. The critical value of $\lambda$, i.e. $\lambda_{\text{critical}}$, is approximately $169$ (\cite{Carlos2012}). The properties of the solutions are presented in table \ref{P1Table9} to table \ref{P1Table10}. We have displayed few graphs [Figure \ref{P1Figure5}] corresponding to some positive values of $\lambda$. 

\noindent{\em \textbf{Case (h):} $\lambda < 0.$}\\
Again, we observe same  remarks  as in subsection \ref{P1typetwo}. In table \ref {P1Table11}, table \ref{P1Table12}  and figure \ref{P1Figure6}  we list all  the numerical data of  approximate solutions.
\section{Tables}\label{P1Tables}
Here, we have listed bellow some numerical data of approximate solutions  corresponding to section \ref{numerics}.
\subsection{Navier boundary condition of type one}
\begin{table}[H]
\caption{\small{Upper solution residue errors  for Navier boundary conditions of type one:}}											
\centering											
\begin{center}											
\resizebox{11cm}{2.2cm}{											
\begin{tabular}	{c   c c c c}										
\hline											
                  											
    $r$     &     $\lambda=0$          &       $\lambda=15$      &       $\lambda=20$       &      $\lambda=31$       \\\hline											
0	&	0	&	0	&	0	&	0	\\
0.1	&	-0.000155912		&	-0.000114429	&	-9.98657E-05	&	-5.76457E-05	\\
0.2	&	-0.000343263	&	-0.000371915	&	-0.000365384	&	-0.000290267	\\
0.3	&	0.001588386		&	0.000620694	&	0.000333768	&	-0.000267199	\\
0.4	&	0.006171169	&		0.003943101	&	0.003114494	&	0.000803931	\\
0.5	&	0.008731215	&	0.007584368	&	0.006742658	&	0.003148673	\\
0.6	&	0.003260515	&	0.00721201	&	0.007721714	&	0.005805328	\\
0.7	&	-0.010610475		&	0.000189979	&	0.003211949	&	0.007033732	\\
0.8	&	-0.026235626		&	-0.011575058	&	-0.006229739	&	0.00548626	\\
0.9	&	-0.035139344	&	-0.02270068	&	-0.01667245	&	0.001231356	\\\hline

\end{tabular}}											
\end{center}											
\label{P1Table1}											
\end{table}
\begin{table}[H]
\caption{\small{Lower solution residue errors  for Navier boundary conditions of type one:}}											
\centering											
\begin{center}											
\resizebox{10cm}{2.2cm}{											
\begin{tabular}	{c   c  c  c   c}										
\hline											
                  											
    $r$     &     $\lambda=0$       &       $\lambda=15$      &       $\lambda=20$   &       $\lambda=31$       \\\hline											
0	&	0	&	0	&	0	&		0	\\
0.1	&	0	&	-8.87083E-06	&	-1.3327E-05	&	-3.45209E-05	\\
0.2	&	0	&	-6.53946E-05	&	-9.45309E-05	&	-0.000206661	\\
0.3	&	0	&	-0.00019217	&	-0.000258526	&	-0.000380835	\\
0.4	&	0	&	-0.000373617	&	-0.000445102	&	-0.00015177	\\
0.5	&	0	&	-0.00056198	&	-0.000542526	&	0.000847379	\\
0.6	&	0	&	-0.000702365	&	-0.000448957	&	0.00257569	\\
0.7	&	0	&	-0.000770031	&	-0.000146034	&	0.004456882	\\
0.8	&	0	&	-0.00081014	&	0.000244377	&	0.005612047	\\
0.9	&	0	&		-0.00097216	&	0.000437111	&	0.005313773	\\\hline

\end{tabular}}											
\end{center}											
\label{P1Table2}											
\end{table}
\begin{table}[H]
\caption{\small{Positive solution residue errors  for Navier boundary conditions of type one:}}											
\centering											
\begin{center}											
\resizebox{11cm}{2.2cm}{											
\begin{tabular}	{c  c  c  c   c}										
\hline											
                  											
    $r$     &     $\lambda=-1$         &       $\lambda=-40$      &       $\lambda=-60$   &         $\lambda=-100$       \\\hline											
0	&	0	&	0	&	0	&	0	\\
0.1	&	-0.000158668	&	-0.000275753	&	-0.000348062	&	-0.000530218	\\
0.2	&	-0.000339023	&	8.87605E-05	&	0.000591477	&	0.002846296	\\
0.3	&	0.001659012	&	0.005084608	&	0.007385895	&		0.012677871	\\
0.4	&	0.006308302	&	0.010250819	&	0.010473629	&		0.001611918	\\
0.5	&	0.00874322	&	0.002958666	&	-0.005421612	&	-0.0359395	\\
0.6	&	0.00289219	&	-0.018916565	&	-0.033821897	&	-0.059343222	\\
0.7	&	-0.011384538	&	-0.040769892	&	-0.05065471	&	-0.040737203	\\
0.8	&	-0.027110199	&	-0.046053985	&	-0.040944067	&	0.006597174	\\
0.9	&	-0.035680636	&	-0.03035376	&	-0.009364712	&	0.056747965	\\\hline

\end{tabular}}											
\end{center}											
\label{P1Table3}											
\end{table}
\begin{table}[H]
\caption{\small{Negative solution residue errors  for Navier boundary conditions of type one:}}											
\centering											
\begin{center}											
\resizebox{11cm}{2.2cm}{											
\begin{tabular}	{c c  c  c   c}										
\hline											
                  											
    $r$     &     $\lambda=-1$       &           $\lambda=-40$      &       $\lambda=-60$   &        $\lambda=-100$       \\\hline											
0	&	0	&	0	&	0	&	0	\\
0.1	&	4.50911E-07	&	1.20523E-05	&	1.55973E-05	&	2.05055E-05	\\
0.2	&	3.62314E-06	&	0.000109664	&	0.000148369	&	0.000209255	\\
0.3	&	1.23135E-05	&	0.000447296	&	0.00064218	&	0.000990024	\\
0.4	&	2.94528E-05	&	0.001327808	&	0.002032183	&	0.003435276	\\
0.5	&	5.81282E-05	&	0.003293586	&	0.005368521	&	0.009921185	\\
0.6	&	0.000101552	&	0.007206178	&	0.012468993	&	0.025072982	\\
0.7	&	0.000162956	&	0.014243773	&	0.026048476	&	0.056683522	\\
0.8	&	0.000245387	&	0.025666456	&	0.049332478	&	0.115399675	\\
0.9	&	0.000351386	&	0.042091333	&	0.08437909	&	0.210205598	\\\hline

\end{tabular}}											
\end{center}											
\label{P1Table4}											
\end{table}
\subsection{Navier boundary condition of type two}
\begin{table}[H]
\caption{\small{Upper solution residue errors  for Navier boundary conditions of type two:}}											
\centering											
\begin{center}											
\resizebox{11cm}{2.2cm}{											
\begin{tabular}	{c  c   c  c c}										
\hline											
                  											
    $r$     &     $\lambda=0$       &            $\lambda=8$      &       $\lambda=10$   &        $\lambda=11.34$       \\\hline											
0	&	0	&	0	&	0	&		0	\\
0.1	&	-3.59128E-05	&	-2.41406E-05	&	-1.98602E-05	&	-1.36854E-05	\\
0.2	&	-0.00019189	&	-0.000144885	&	-0.000124646	&		-9.1929E-05	\\
0.3	&	-0.000229253	&	-	-0.00026456	&	-0.000257591	&		-0.000223497	\\
0.4	&	0.000379522	&		-7.65776E-05	&	-0.000194005	&		-0.000294198	\\
0.5	&	0.001951425	&		0.00073382	&	0.000336884	&		-0.000128778	\\
0.6	&	0.004095042	&		0.002209409	&	0.001454893	&		0.000417341	\\
0.7	&	0.005675344	&		0.003964315	&	0.002983797	&		0.001363933	\\
0.8	&	0.005331432	&		0.005273065	&	0.004455361	&		0.00255888	\\
0.9	&	0.002218034	&		0.005369172	&	0.005271528	&		0.003704074	\\\hline

\end{tabular}}											
\end{center}											
\label{P1Table5}											
\end{table}

\begin{table}[H]											
\caption{\small{Lower solution residue errors  for Navier boundary conditions of type two:}}											
\centering											
\begin{center}											
\resizebox{10cm}{2.2cm}{											
\begin{tabular}{c  c  c  c   c}											
\hline											
                  											
    $r$     &     $\lambda=0$       &             $\lambda=8$      &       $\lambda=10$   &          $\lambda=11.34$       \\\hline											
0	&	0		&	0	&	0	&		0	\\
0.1	&	0	&		-5.31802E-06	&	-8.03463E-06	&	-1.31725E-05	\\
0.2	&	0	&		-3.95382E-05	&	-5.77125E-05	&	-8.90043E-05	\\
0.3	&	0	&		-0.00011771	&	-0.000161094	&	-0.000219188	\\
0.4	&	0	&	-0.000232269	&	-0.00028476	&	-0.000297926	\\
0.5	&	0	&	-0.000352779	&	-0.00035549	&	-0.000159086	\\
0.6	&	0	&	-0.000434856	&	-0.00028637	&	0.000339477	\\
0.7	&	0	&	-0.000436068	&	-1.38768E-05	&	0.001227054	\\
0.8	&	0	&	-0.00033461	&	0.000467775	&	0.002372209	\\
0.9	&	0	&	-0.000146031	&	0.001087285	&	0.003501342		\\\hline
\end{tabular}}											
\end{center}
\label{P1Table6}																		
\end{table}
\begin{table}[H]
\caption{\small{Positive solution residue errors  for Navier boundary conditions of type two:}}									
\centering											
\begin{center}											
\resizebox{11cm}{2.2cm}{											
\begin{tabular}	{  c  c  c c c}										
\hline											
                  											
    $r$     &     $\lambda=-1$       &       $\lambda=-50$      &       $\lambda=-100$   &       $\lambda=-160$       \\\hline											
0	&	0		&	0	&	0	&		0	\\
0.1	&	-3.71856E-05	&		-9.52198E-05	&	-0.000171484	&		-0.0003356	\\
0.2	&	-0.000196294		&	-0.000293491	&	-0.000196499	&		0.000690227	\\
0.3	&	-0.000221431	&		0.000634465	&	0.002515105	&		0.007458185	\\
0.4	&	0.000437305	&		0.003624522	&	0.007340292	&		0.009598758	\\
0.5	&	0.002086313	&		0.006721547	&	0.006882551	&		-0.008688033	\\
0.6	&	0.004270952	&		0.005539989	&	-0.00553502	&		-0.039604188	\\
0.7	&	0.005768587	&		-0.002998379	&	-0.026584976	&		-0.055617692	\\
0.8	&	0.005166738	&		-0.017397284	&	-0.043629852	&		-0.037115791	\\
0.9	&	0.001651162	&		-0.031769072	&	-0.043368929	&		0.015496504	\\\hline
								
\end{tabular}}											
\end{center}											
\label{P1Table7}											
\end{table}

\begin{table}[H]
\caption{\small{Negative solution residue errors  for Navier boundary conditions of type two:}}											
\centering											
\begin{center}											
\resizebox{11cm}{2.2cm}{											
\begin{tabular}	{  c  c  c c  c}										
\hline											
                  											
    $r$     &     $\lambda=-1$       &              $\lambda=-50$      &       $\lambda=-100$   &          $\lambda=-160$       \\\hline											
0	&	0	&		0	&	0	&		0	\\
0.1	&	0	&		1.0114E-05	&	1.37015E-05	&		1.59205E-05	\\
0.2	&	0	&		9.66798E-05	&	0.000143445	&		0.000179917	\\
0.3	&	1.21638E-05	&		0.000422316	&	0.000702496	&		0.000963885	\\
0.4	&	2.92389E-05	&		0.001356286	&	0.002538349	&		0.003806251	\\
0.5	&	5.80708E-05	&		0.003662489	&	0.007683917	&		0.012541867	\\
0.6	&	0.000102229	&		0.008777896	&	0.020536832	&		0.036345698	\\
0.7	&	0.000165517	&		0.019165906	&	0.049731562	&		0.095077006	\\
0.8	&	0.000251823	&		0.038627935	&	0.110500915	&		0.227275426	\\
0.9	&	0.000364852	&		0.072242013	&	0.226171011	&		0.497848871	\\\hline

\end{tabular}}											
\end{center}											
\label{P1Table8}											
\end{table}
\subsection{Dirichlet boundary condition}
\begin{table}[H]											
\caption{\small{Lower solution residue errors  for Dirichlet  boundary conditions:}}											
\centering											
\begin{center}											
\resizebox{10cm}{2.2cm}{											
\begin{tabular}{  c  c  c  c  c}											
\hline											
                  											
    $r$     &     $\lambda=0$       &             $\lambda=100$      &       $\lambda=150$   &          $\lambda=168.5$       \\\hline											
0	&	0		&	0	&	0	&		0	\\
0.1	&	0	&		-0.000206385	&	-0.000456193	&		-0.000792126	\\
0.2	&	0	&		-0.001202861	&	-0.00172148	&		-0.001148293	\\
0.3	&	0	&		-0.002150566	&	0.000877891	&		0.009654358	\\
0.4	&	0	&		-0.001088836	&	0.011254878	&		0.028840292	\\
0.5	&	0	&		0.002659014	&	0.023546219	&		0.03096701	\\
0.6	&	0	&		0.00643131	&	0.025416409	&		0.002113967	\\
0.7	&	0	&		0.00420898	&	0.011384789	&		-0.032940028	\\
0.8	&	0	&		-0.012205803	&	-0.008754913	&		-0.02467811	\\
0.9	&	0	&		-0.053923736	&	-0.012406316	&		0.083542791	\\\hline
									
\end{tabular}}											
\end{center}											
\label{P1Table9}											
\end{table}
\begin{table}[H]											
\caption{\small{Upper solution residue errors  for Dirichlet  boundary conditions:}}											
\centering											
\begin{center}											
\resizebox{11cm}{2.2cm}{											
\begin{tabular}{c  c  c  c   c}											
\hline											
                  											
    $r$     &     $\lambda=0$       &            $\lambda=100$      &       $\lambda=150$   &           $\lambda=168.5$       \\\hline											
0	&	0		&	0	&	0	&		0	\\
0.1	&	-0.002203063	&		-0.001886917	&	-0.001342885	&		-0.000877685	\\
0.2	&	0.033797318	&		0.011517037	&	0.002749495	&		-0.000788042	\\
0.3	&	0.030922362	&		0.044532539	&	0.02827772	&		0.012361223	\\
0.4	&	-0.014345268	&		-0.003159996	&	0.035900291		&	0.03213499	\\
0.5	&	-0.022139308		&	-0.012445316	&	-0.020163177	&		0.027844336	\\
0.6	&	-0.060894281	&		-0.165508967	&	-0.095517516	&		-0.01052225	\\
0.7	&	0.016518058	&		-0.068424986	&	-0.010280418	&		-0.047258847	\\
0.8	&	0.076289167	&		0.097795772	&	-0.004867326	&		-0.026755307	\\
0.9	&	0.756076643	&		0.455082275	&	0.218123002	&		0.107541122	\\\hline
											
\end{tabular}}											
\end{center}											
\label{P1Table10}											
\end{table}
\begin{table}[H]
\caption{\small{ Negative solution residue errors   for Dirichlet  boundary conditions:}}											
\centering											
\begin{center}											
\resizebox{11cm}{2.2cm}{											
\begin{tabular}	{  c  c  c  c  c}										
\hline											
                  											
    $r$     &     $\lambda=-1$       &              $\lambda=-10$      &       $\lambda=-15$   &         $\lambda=-25$       \\\hline											
0	&	0		&	0	&	0	&		0	\\
0.1	&	1.37E-06		&	1.32986E-05	&	1.96493E-05	&		3.18005E-05	\\
0.2	&	1.09624E-05	&		0.000108594	&	0.00016204	&		0.000267254	\\
0.3	&	3.71163E-05	&		0.000377932	&	0.00057214	&		0.000969753	\\
0.4	&	8.8313E-05	&		0.000928682	&	0.001429442	&		0.002498697	\\
0.5	&	0.000173129	&		0.001878535	&	0.002937946	&		0.005288295	\\
0.6	&	0.000300006	&		0.003333426	&	0.005277309	&		0.009716634	\\
0.7	&	0.000476799	&		0.005340493	&	0.008494767	&		0.015794103	\\
0.8	&	0.000710047	&		0.00780818	&	0.012317241	&		0.022597391	\\
0.9	&	0.0010039	&		0.010388382	&	0.015872278	&		0.027416471	\\\hline
								
\end{tabular}}											
\end{center}											
\label{P1Table11}											
\end{table}

\begin{table}[H]
\caption{\small{ Positive solution residue errors   for Dirichlet  boundary conditions:}}											
\centering											
\begin{center}											
\resizebox{11cm}{2.2cm}{											
\begin{tabular}	{c  c  c    c  c}										
\hline											
                  											
    $r$     &     $\lambda=-1$       &             $\lambda=-10$      &       $\lambda=-15$   &         $\lambda=-25$       \\\hline											
0	&	0		&	0	&	0	&		0	\\
0.1	&	-0.002200389	&		-0.002169851	&	-0.002147706	&		-0.00209184	\\
0.2	&	0.034043159	&		0.036262321	&	0.03749797	&	0.039966042	\\
0.3	&	0.030454157	&		0.025927616	&	0.023170528	&		0.017144747	\\
0.4	&	-0.144959813		&	-0.158409209	&	-0.165767258	&		-0.18015219	\\
0.5	&	-0.022137696		&	-0.02202973	&	-0.021897521	&		-0.02148112	\\
0.6	&	-0.058843725		&	-0.039859043	&	-0.028938252	&		-0.006427889	\\
0.7	&	0.016822394	&		0.019683635	&	0.021384849	&		0.025094771	\\
0.8	&	0.078154323	&		0.097632679	&	0.11090339	&		0.144151755	\\
0.9	&	0.802566715	&		0.109459893	&	0.130400343	&	0.186695393	\\\hline
						
\end{tabular}}											
\end{center}											
\label{P1Table12}											
\end{table}
\section{Figures}\label{P1Figures}
Here we have placed below approximate solutions graph corresponding to different types of boundary condition.
\subsection{Navier boundary condition of type one}
\begin{figure}[H]
\centering
\subfigure[\,\,$\lambda=0$]{\includegraphics[width=.45\linewidth]{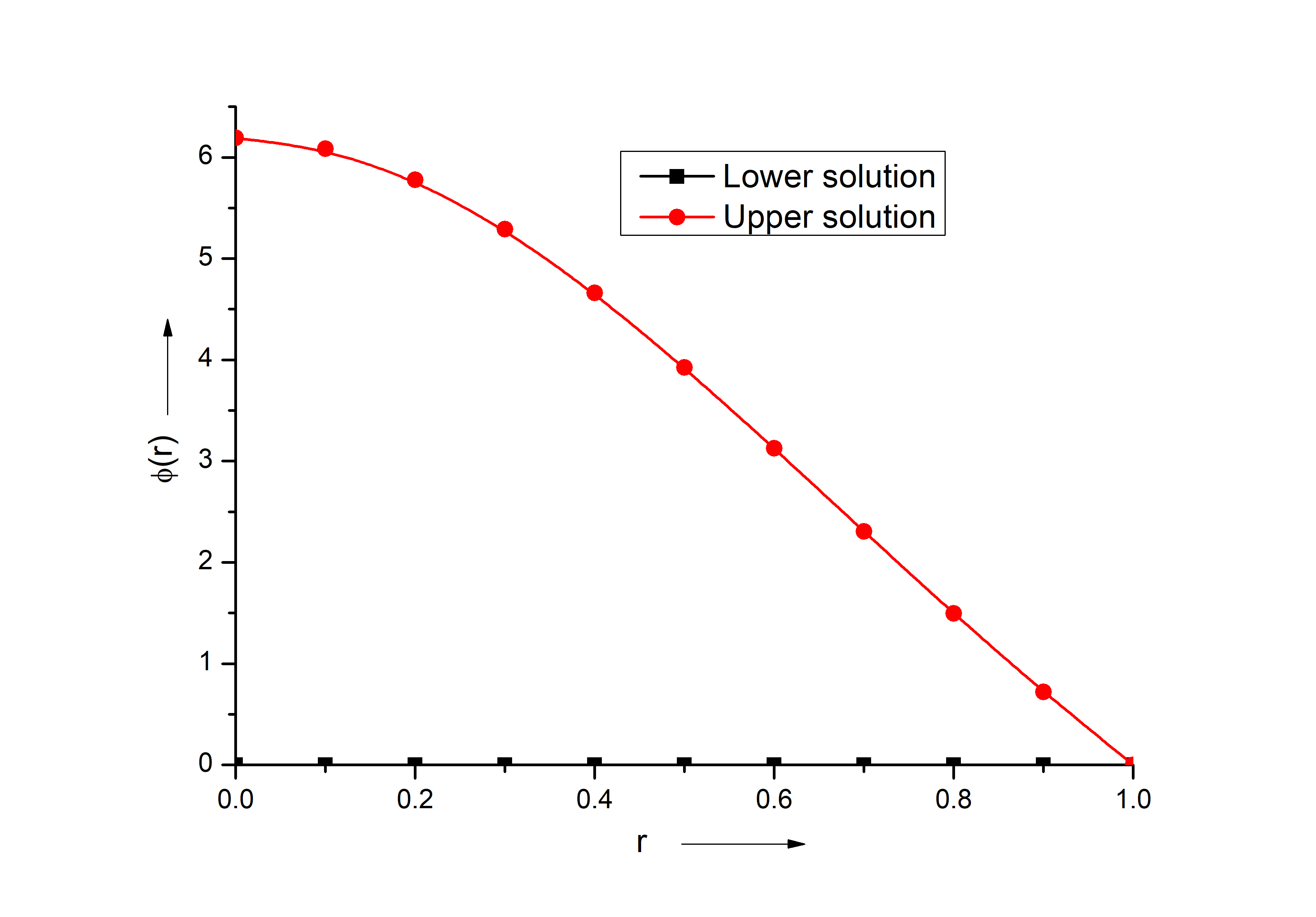}} 
\subfigure[\,\,$\lambda=15$]{\includegraphics[width=.45\linewidth]{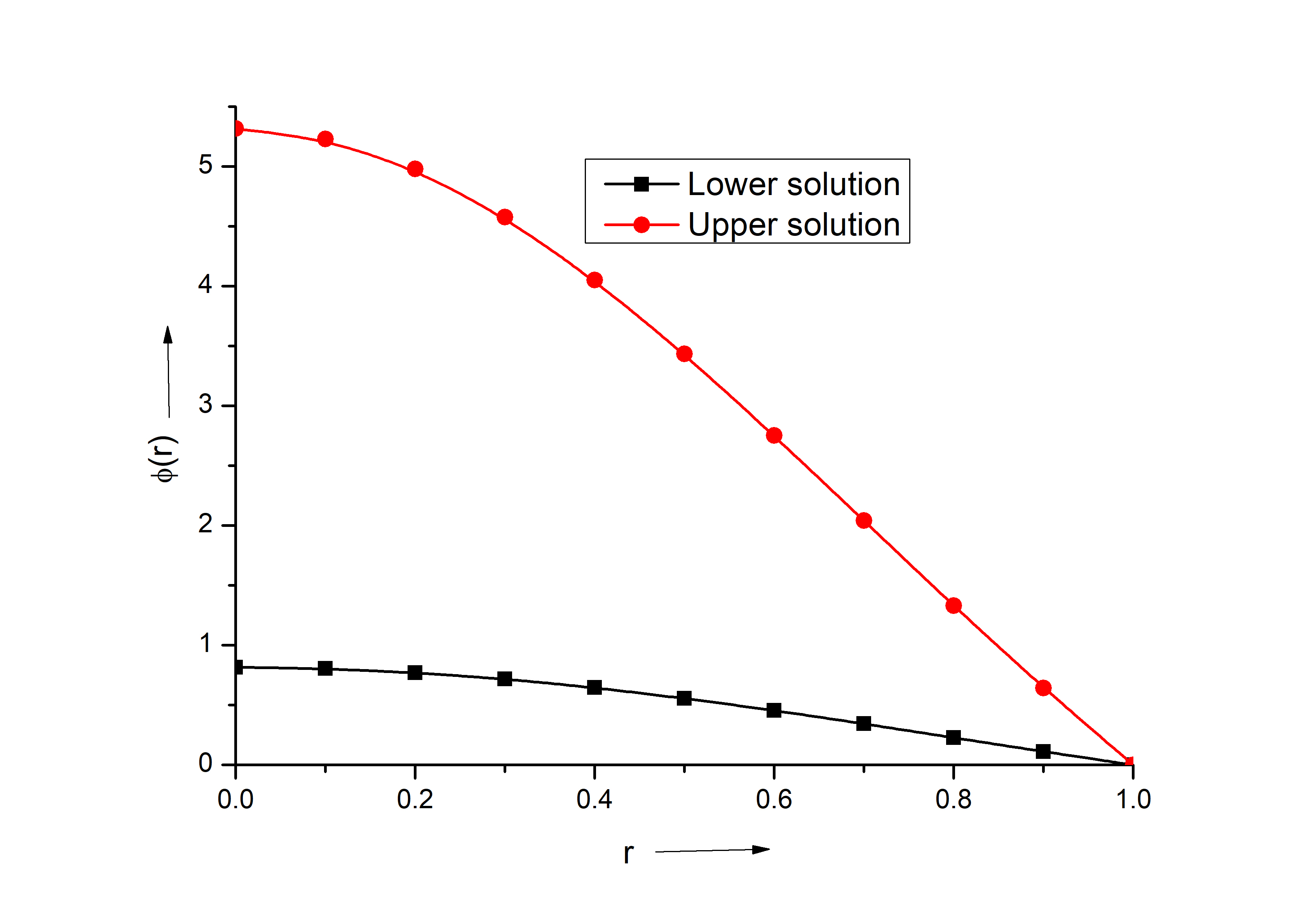}}  
\end{figure}
\begin{figure}[H]
\centering
\subfigure[\,\,$\lambda=20$]{\includegraphics[width=.45\linewidth]{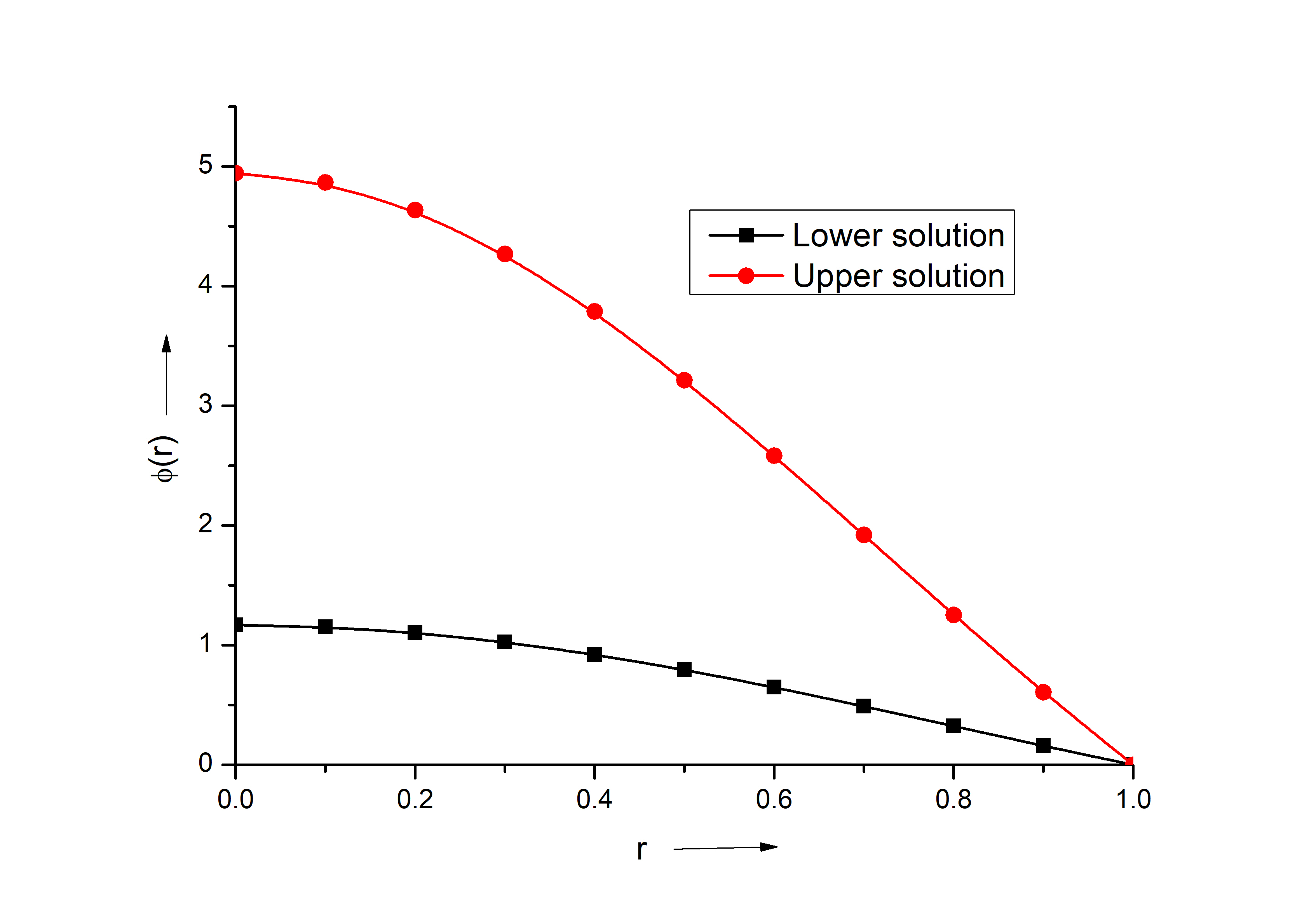}}  
\subfigure[\,\,$\lambda=31$]{\includegraphics[width=.45\linewidth]{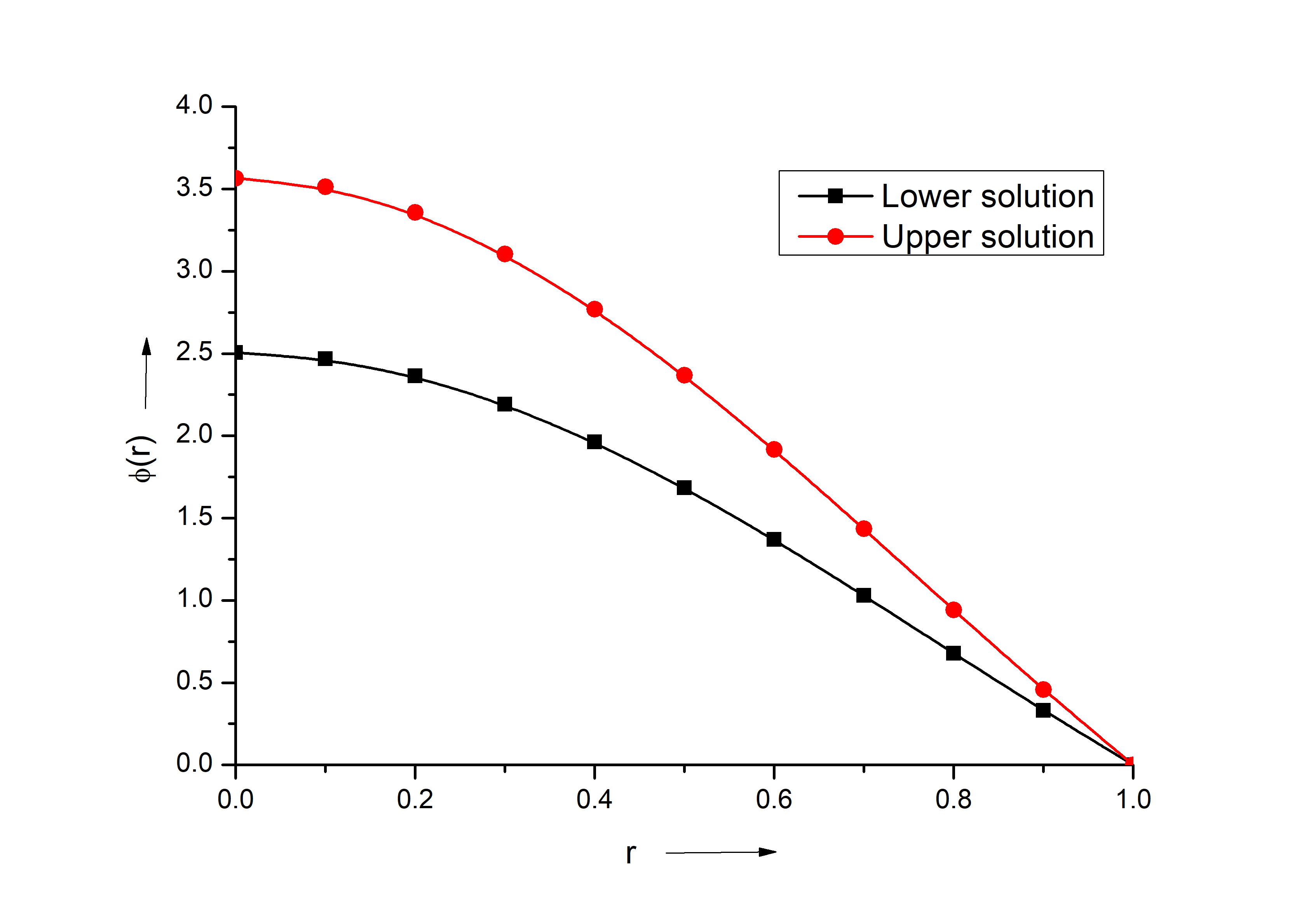}} 
\caption{Graph of $\phi(r)$ versus $r$ for positive $\lambda$.}
\label{P1Figure1}
\end{figure}

\begin{figure}[H]
\centering
\subfigure[\,\,$\lambda=-1$]{\includegraphics[width=.45\linewidth]{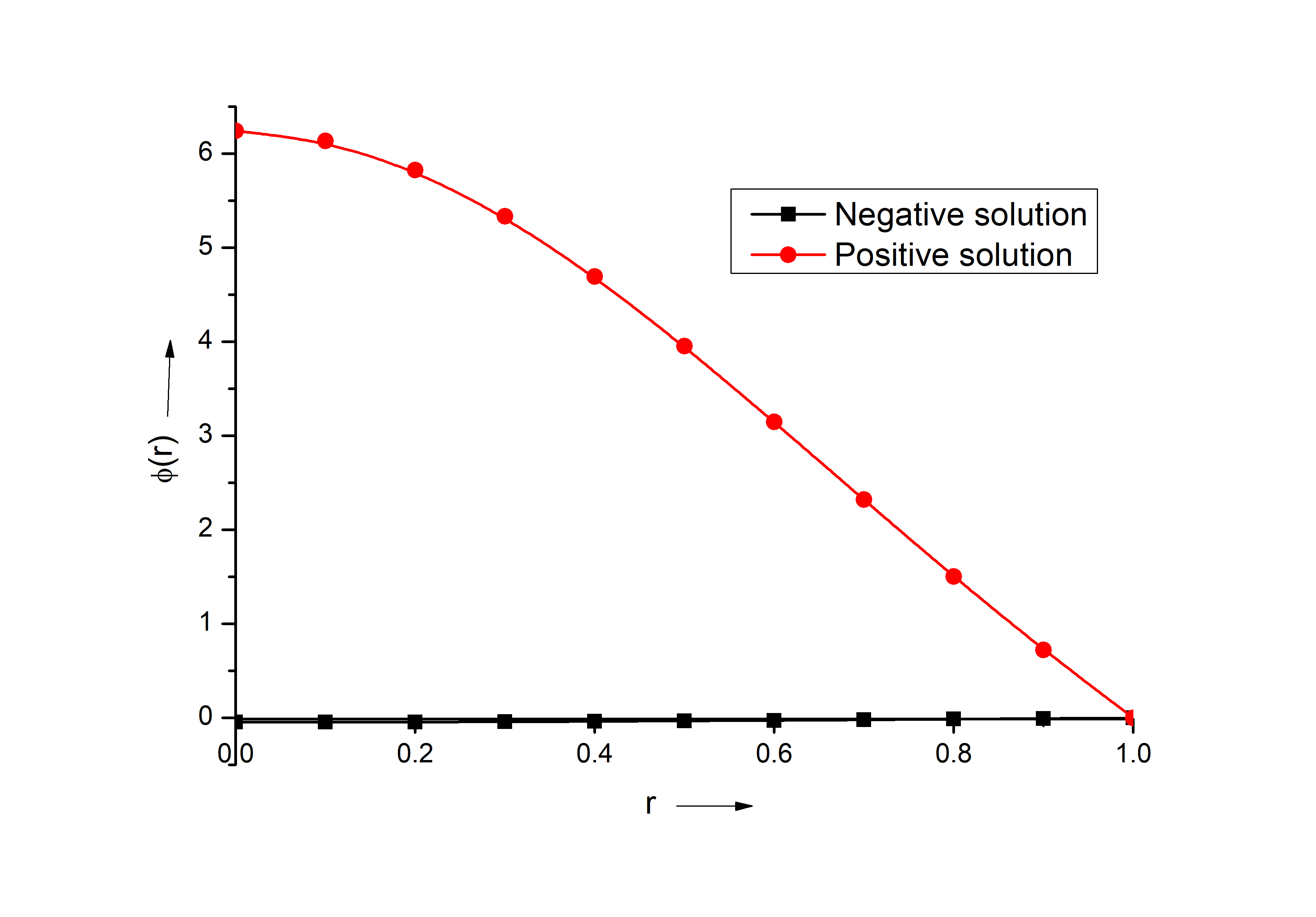}} 
\subfigure[\,\,$\lambda=-40$]{\includegraphics[width=.45\linewidth]{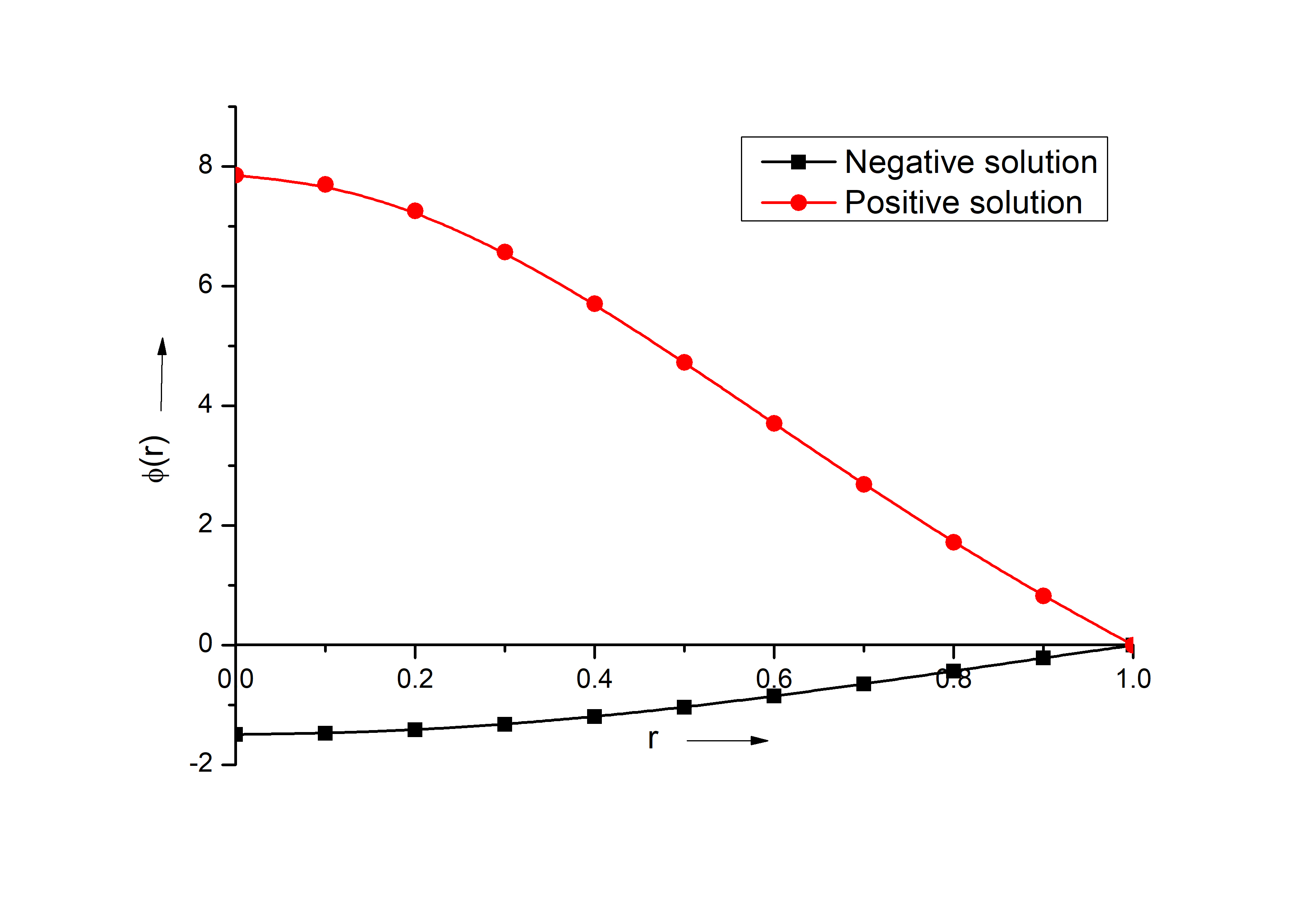}}
\end{figure}
\begin{figure}[H]
\centering
\subfigure[\,\,$\lambda=-60$]{\includegraphics[width=.45\linewidth]{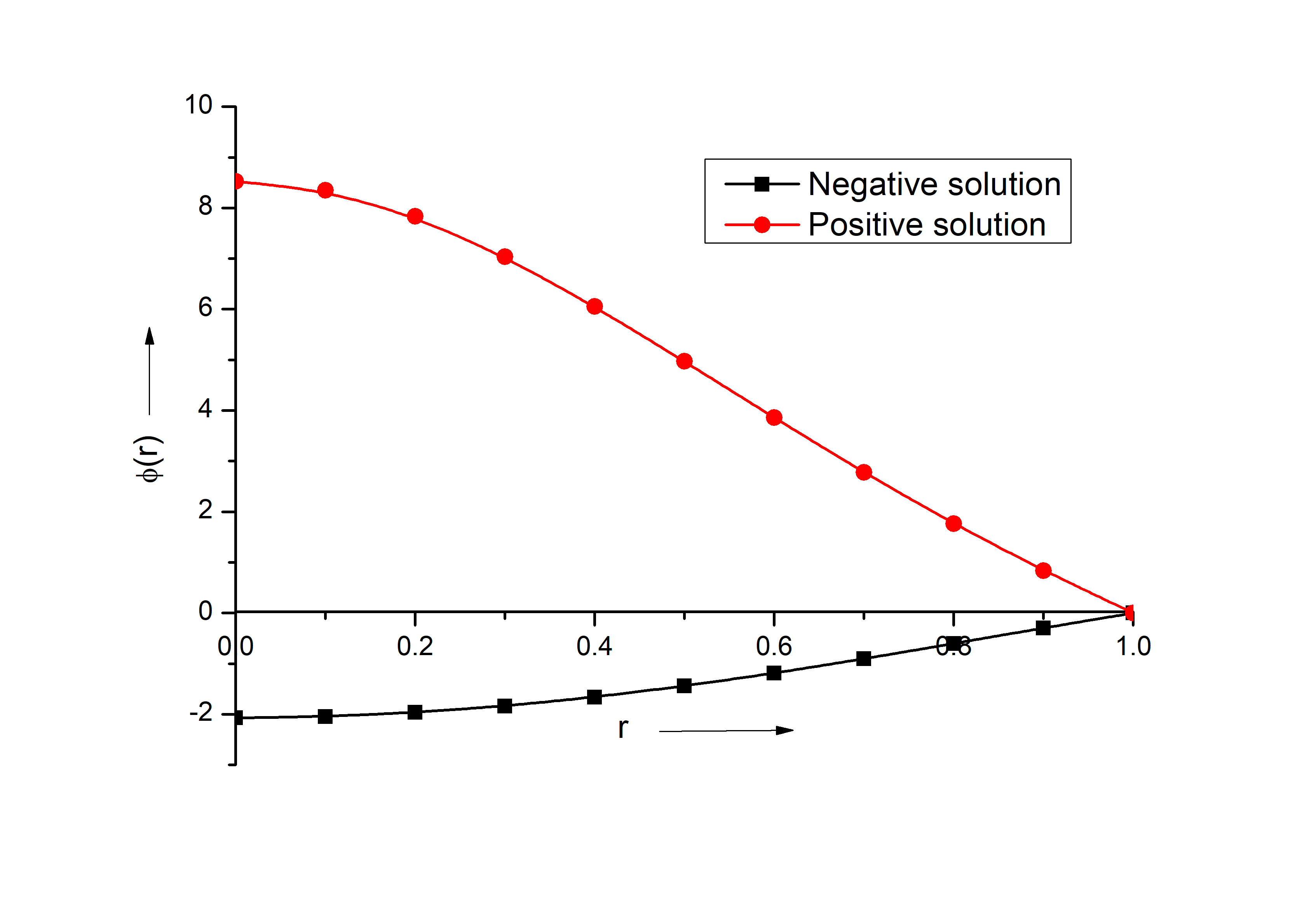}} 
\subfigure[\,\,$\lambda=-100$]{\includegraphics[width=.45\linewidth]{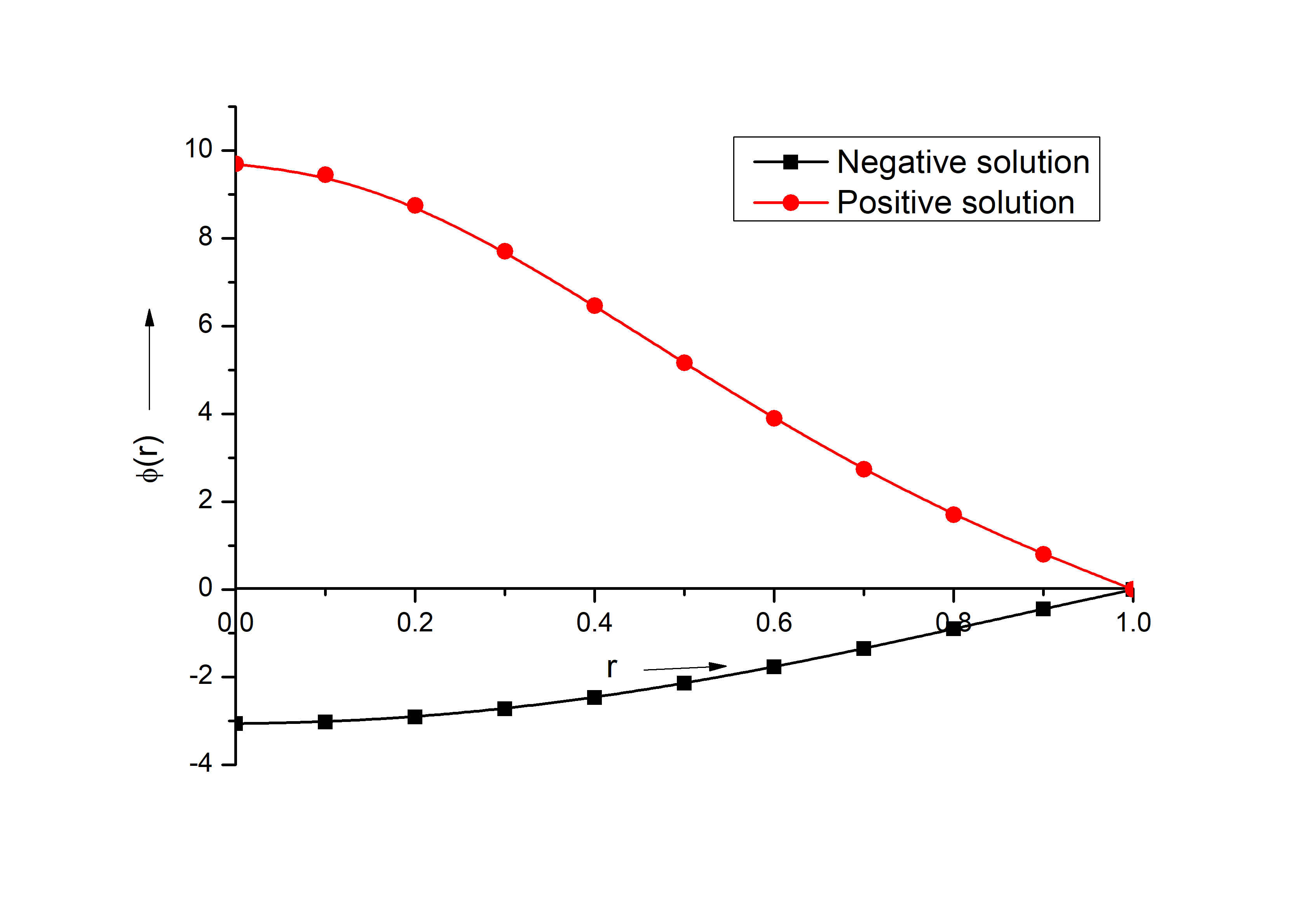}} 
\caption{Graph of $\phi(r)$ versus $r$ for negative $\lambda$.}
\label{P1Figure2}
\end{figure}
\subsection{Navier boundary condition of type two}
\begin{figure}[H]
\centering
\subfigure[\,\,$\lambda=0$]{\includegraphics[width=.45\linewidth]{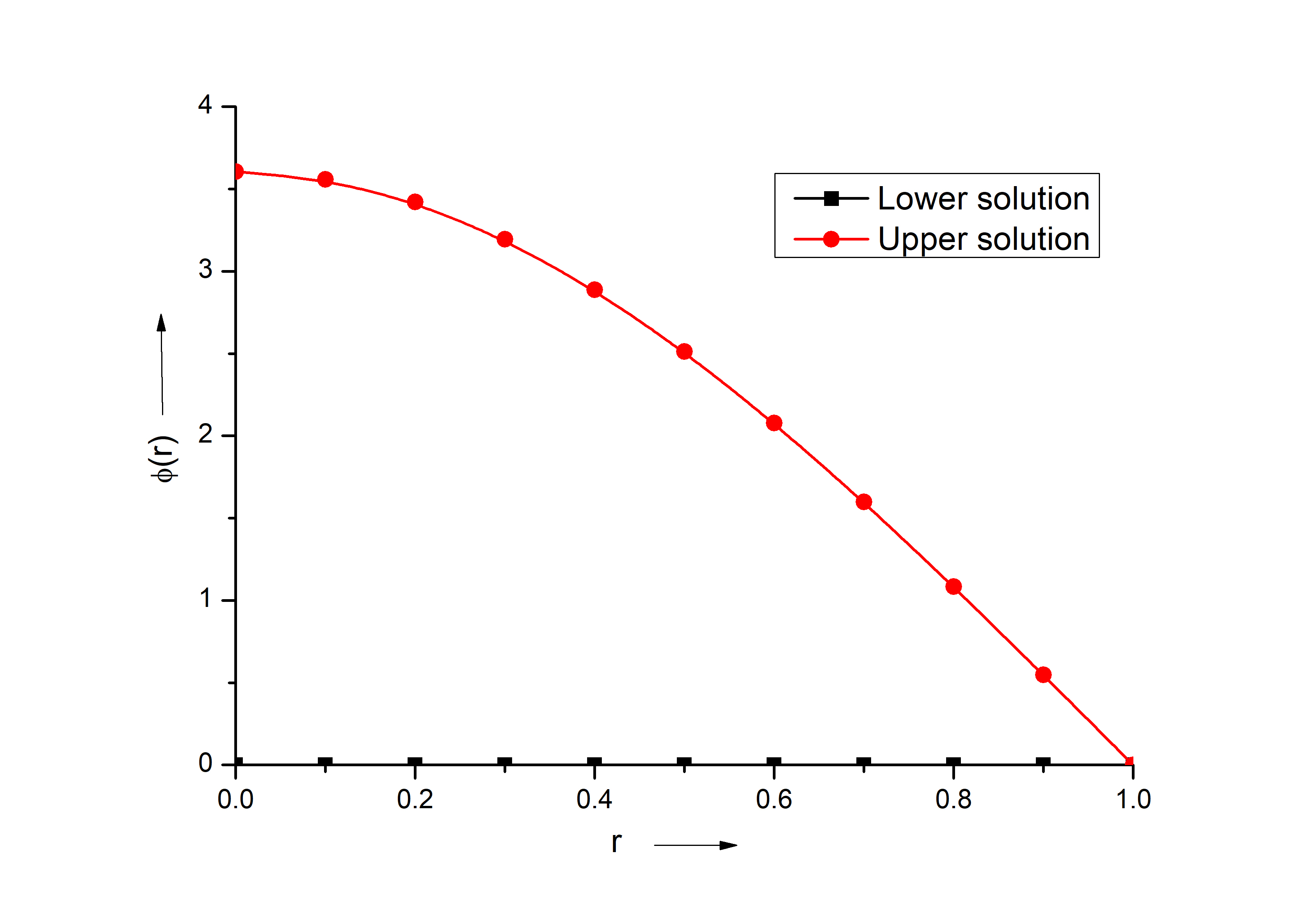}}  
\subfigure[\,\,$\lambda=8$]{\includegraphics[width=.45\linewidth]{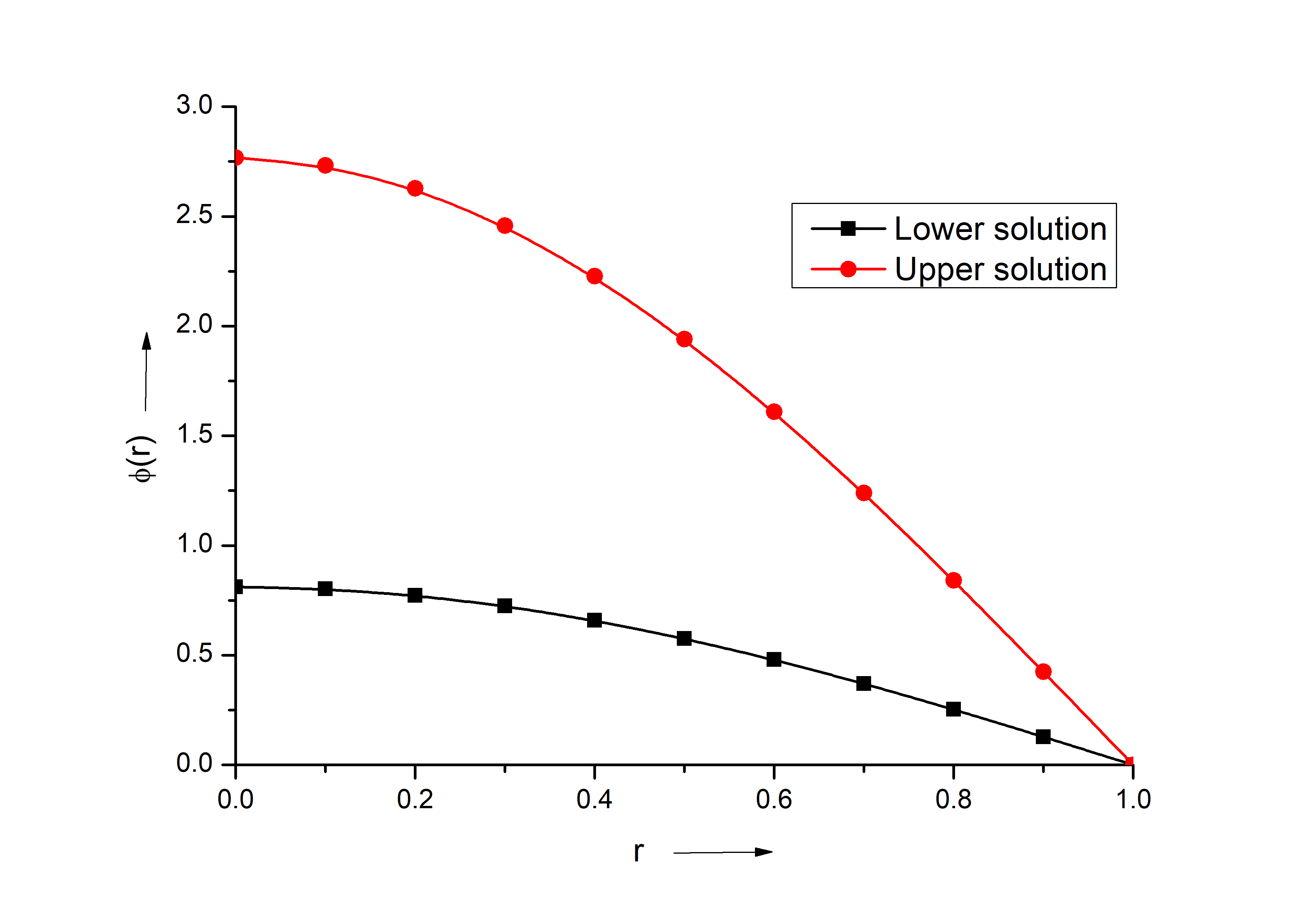}}  
\end{figure}
\begin{figure}[H]
\centering
\subfigure[\,\,$\lambda=10$]{\includegraphics[width=.45\linewidth]{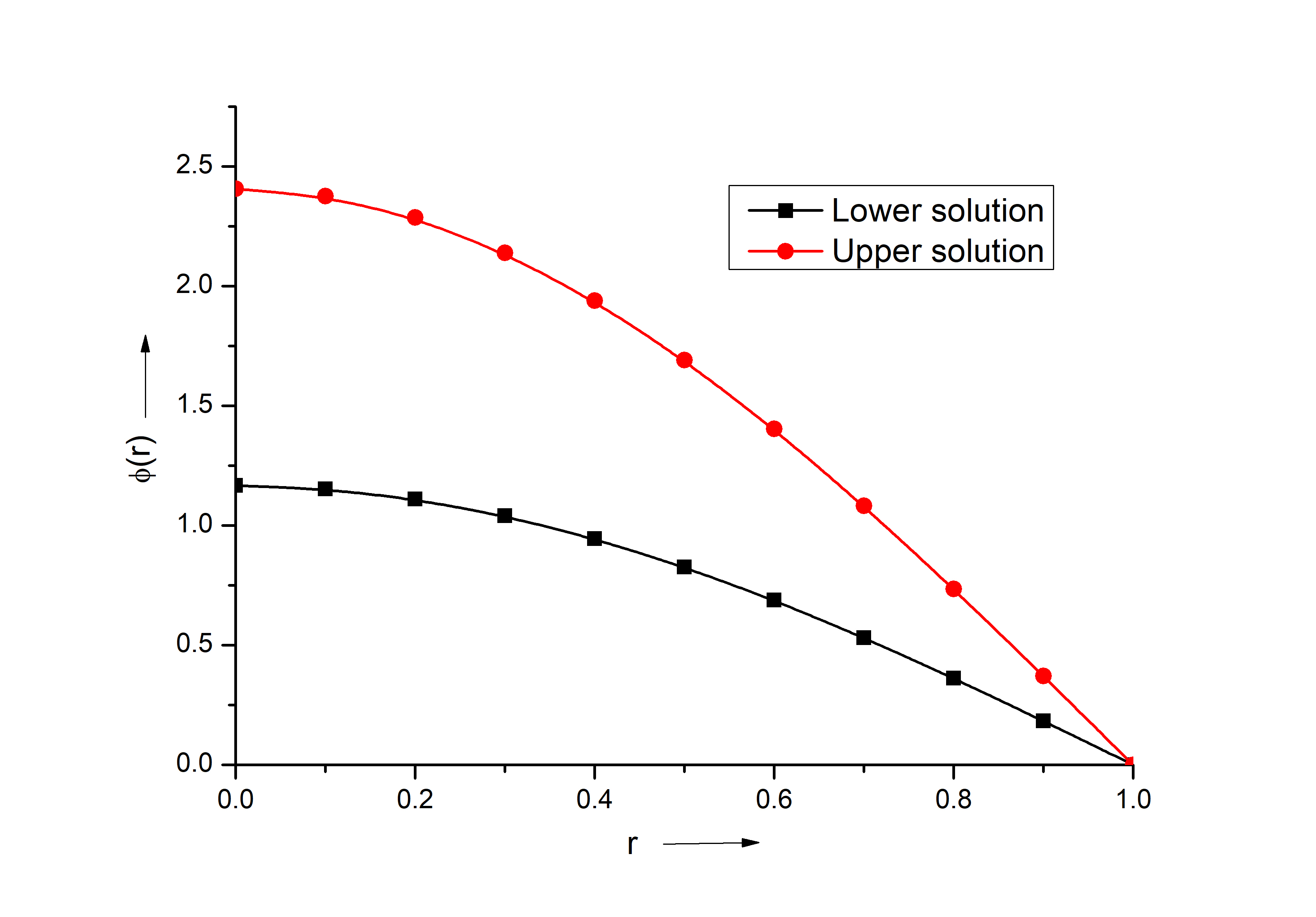}} 
\subfigure[\,\,$\lambda=11.34$]{\includegraphics[width=.45\linewidth]{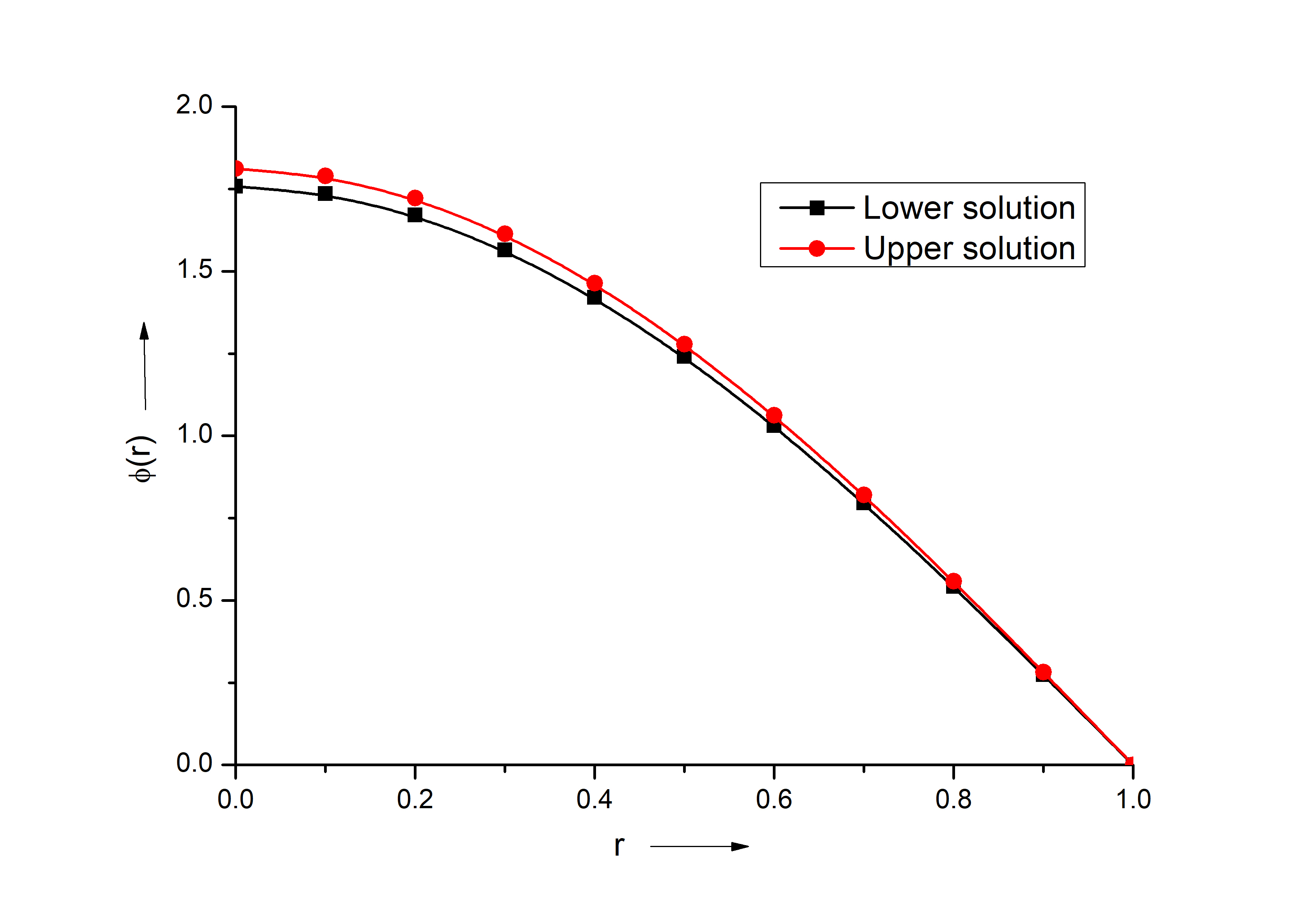}} 
\caption{Graph of $\phi(r)$ versus $r$ for positive $\lambda$.}
\label{P1Figure3}
\end{figure}

\begin{figure}[H]
\centering
\subfigure[\,\,$\lambda=-1$]{\includegraphics[width=.45\linewidth]{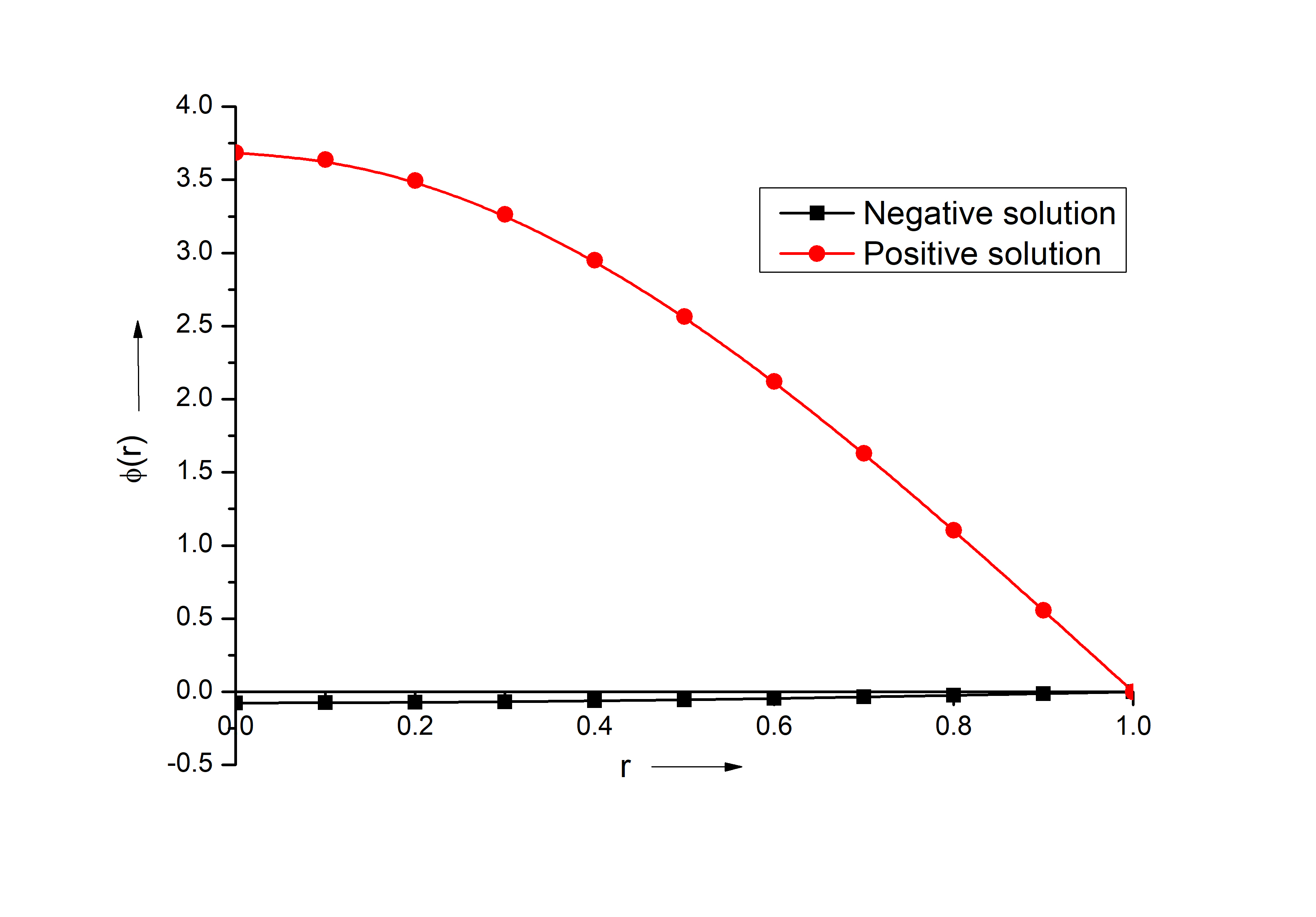}}  
\subfigure[\,\,$\lambda=-50$]{\includegraphics[width=.45\linewidth]{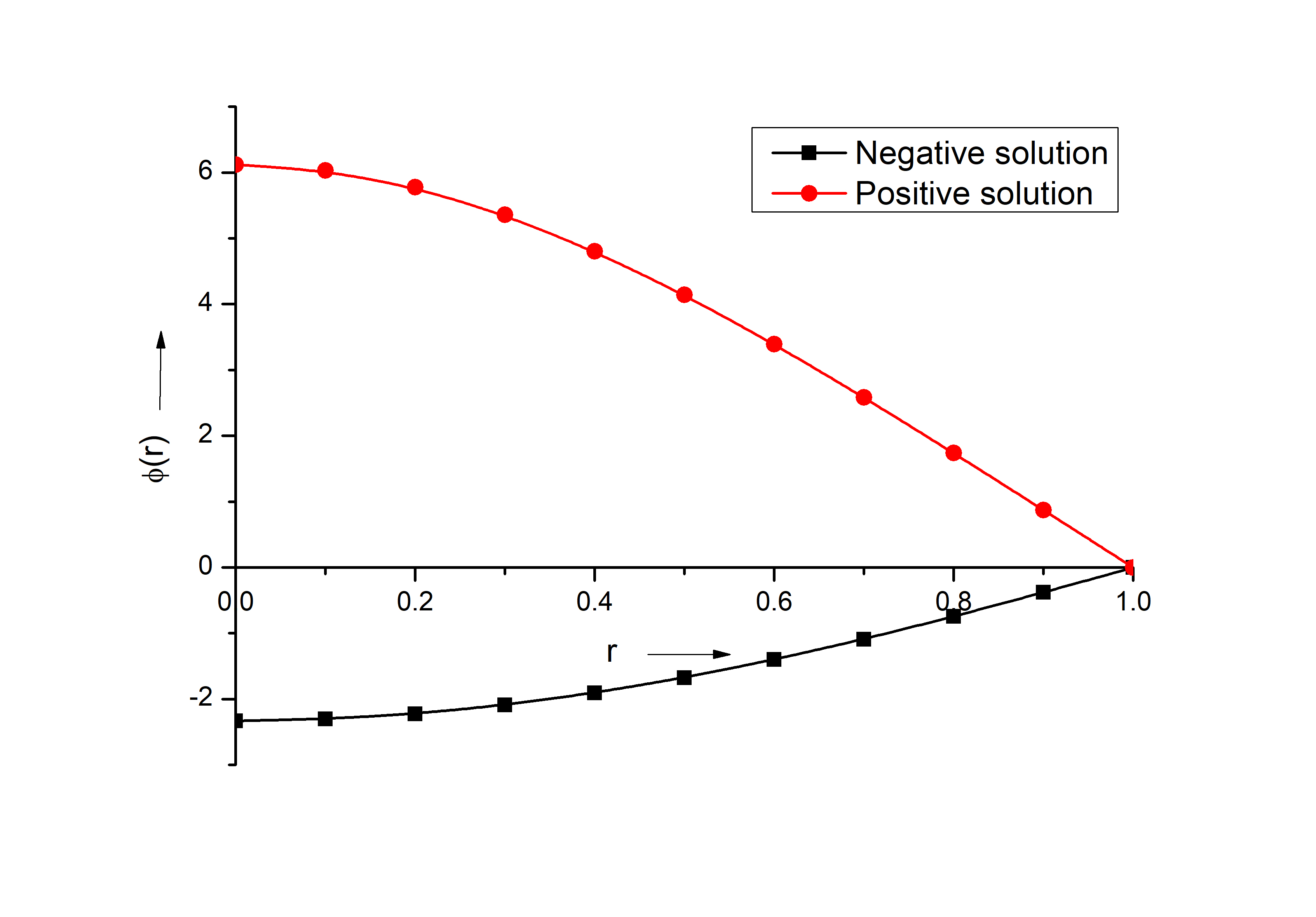}}
\end{figure}
\begin{figure}[H]
\centering
\subfigure[\,\,$\lambda=-100$]{\includegraphics[width=.45\linewidth]{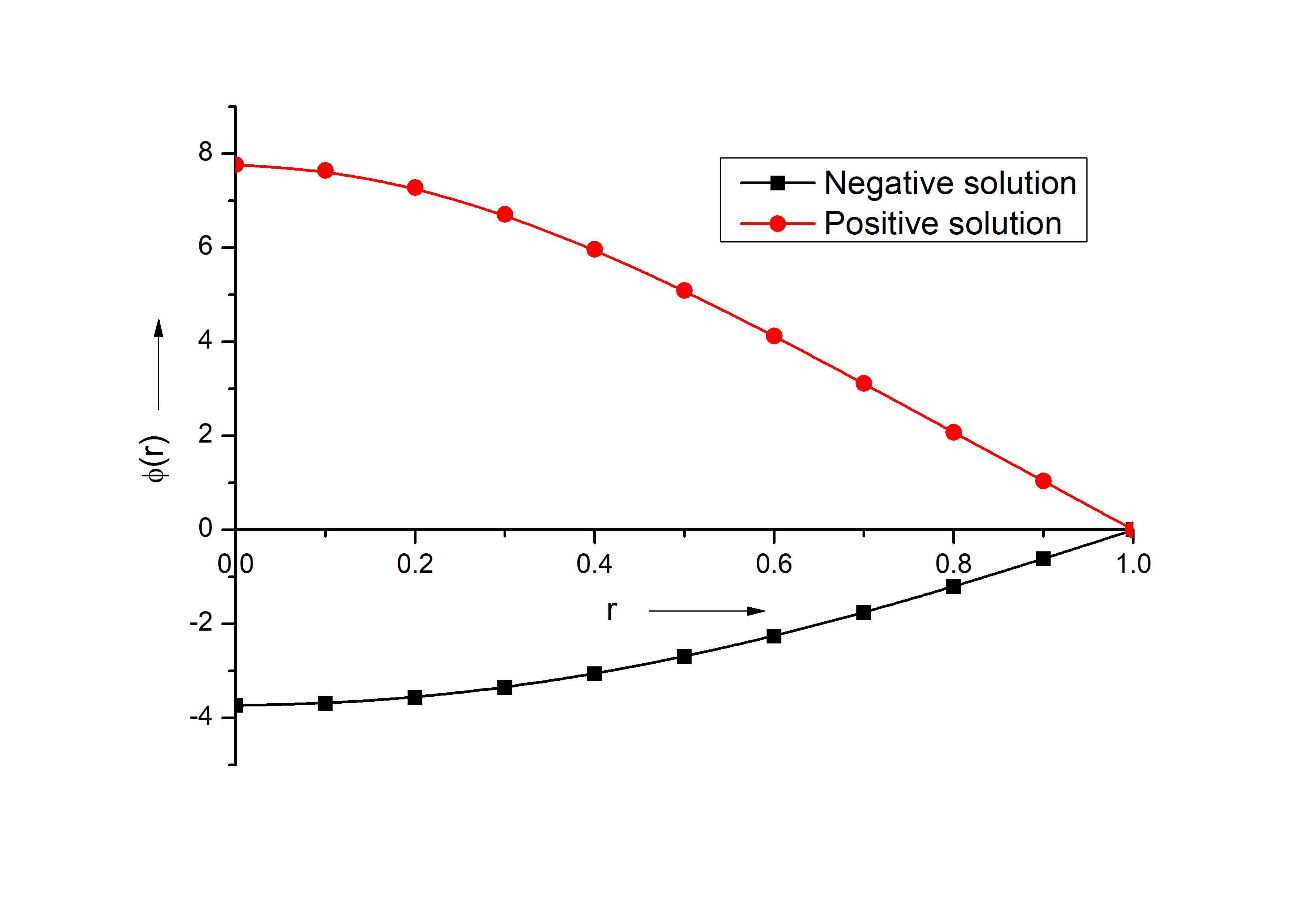}}  
\subfigure[\,\,$\lambda=-160$]{\includegraphics[width=.45\linewidth]{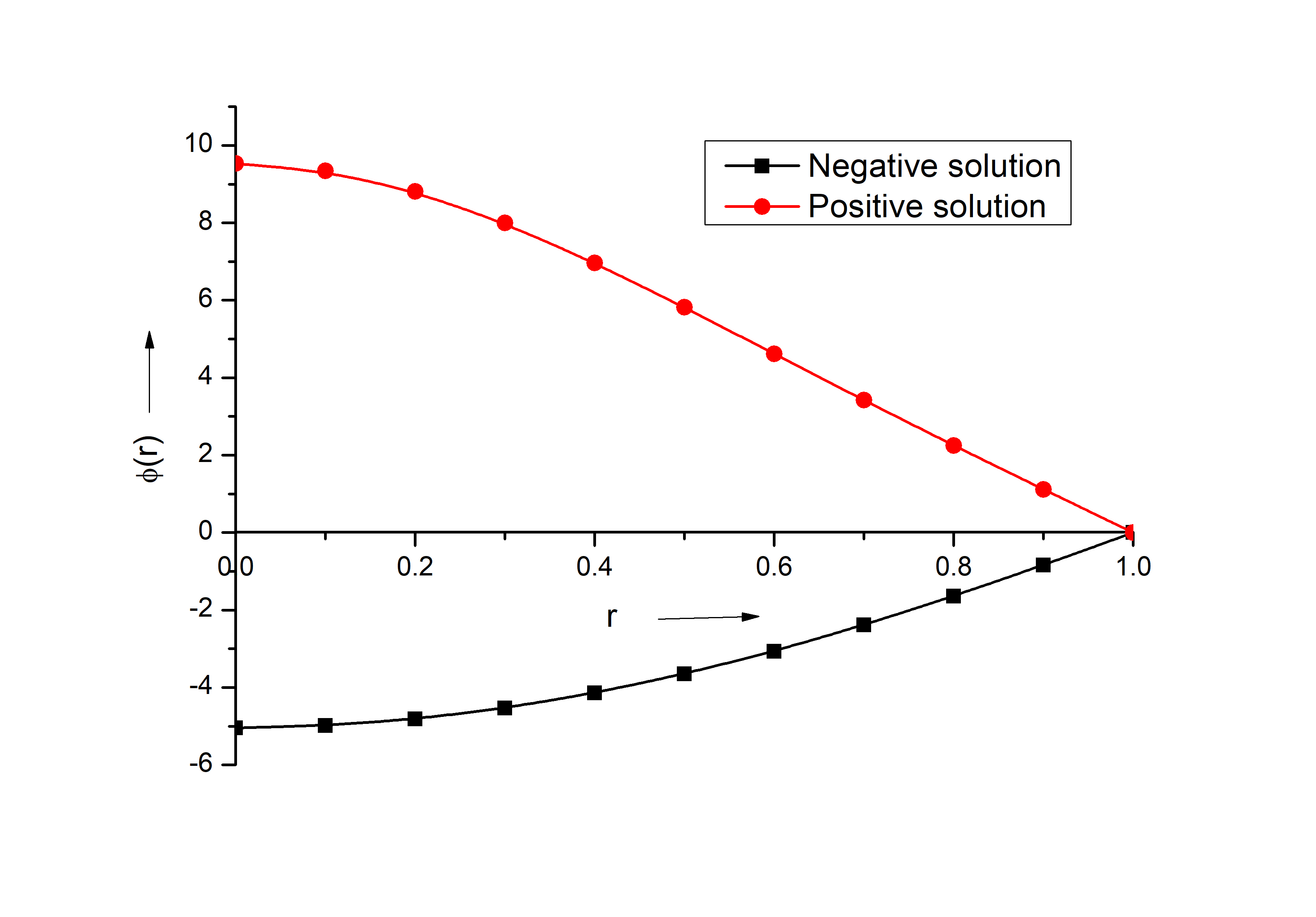}} 
\caption{Graph of $\phi(r)$ versus $r$ for negative $\lambda$.}
\label{P1Figure4}
\end{figure}

\subsection{Dirichlet boundary condition}
\begin{figure}[H]
\centering
\subfigure[\,\,$\lambda=0$]{\includegraphics[width=.45\linewidth]{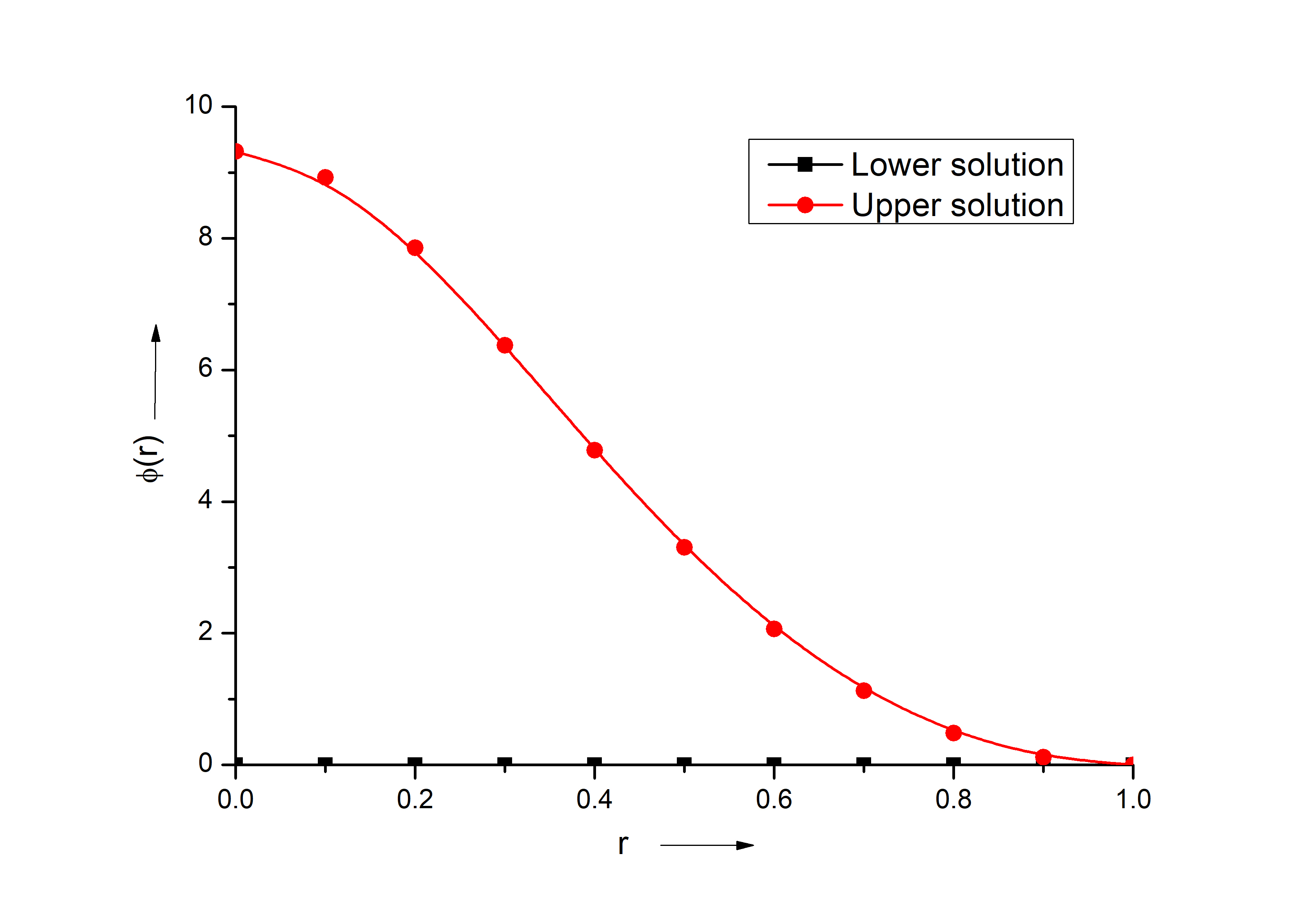}}  
\subfigure[\,\,$\lambda=100$]{\includegraphics[width=.45\linewidth]{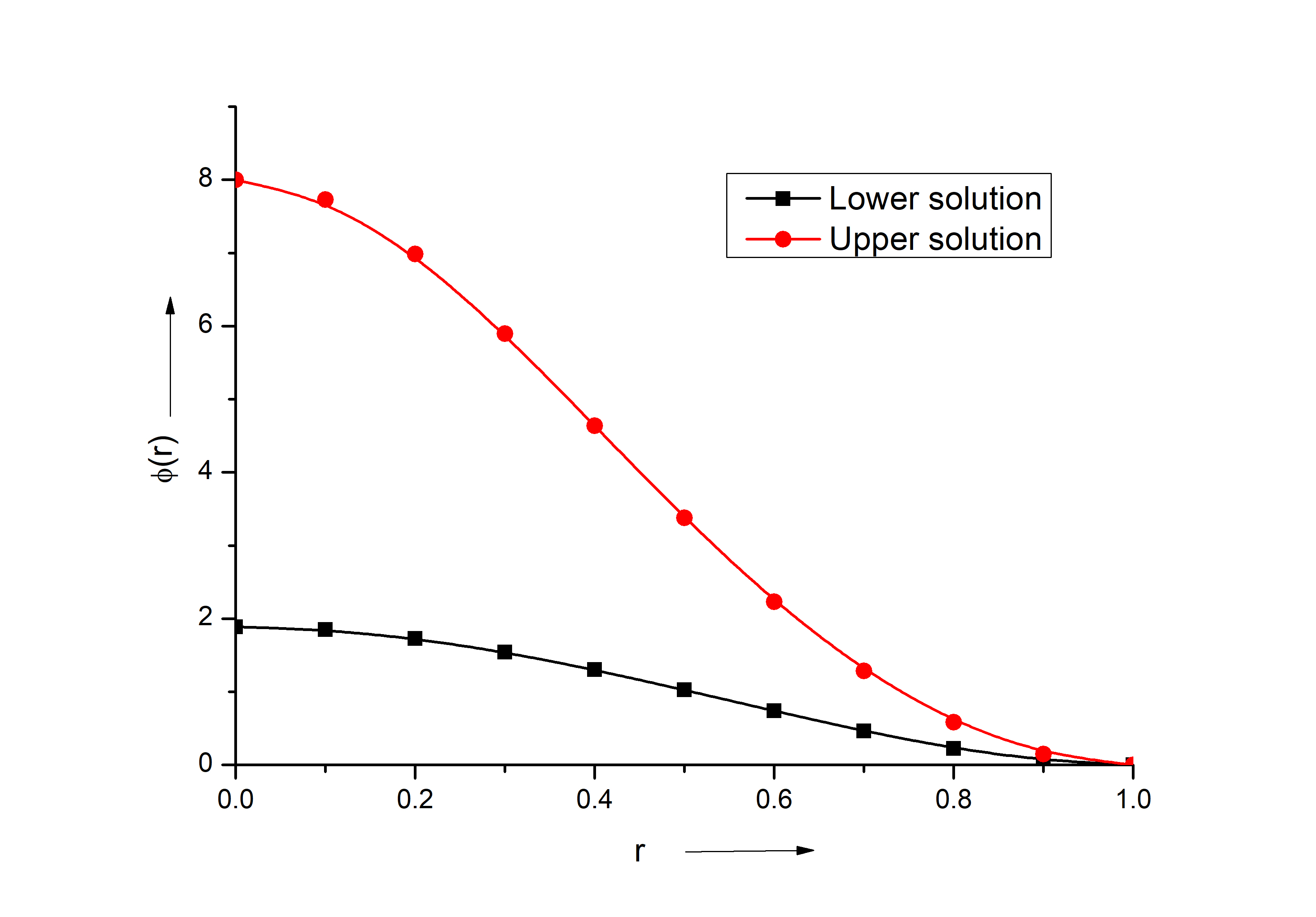}}  
\end{figure}
\begin{figure}[H]
\centering
\subfigure[\,\,$\lambda=150$]{\includegraphics[width=.45\linewidth]{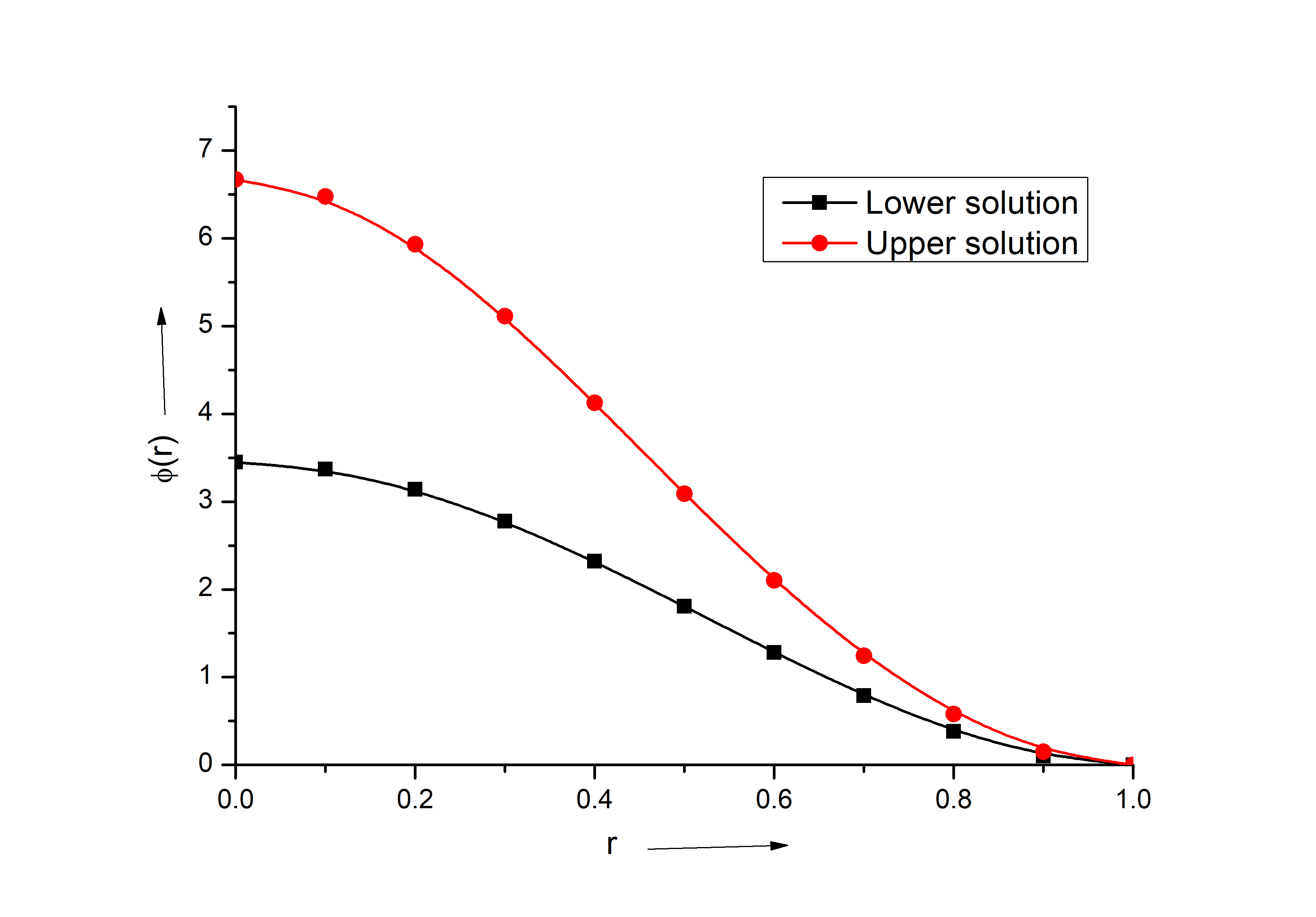}}  
\subfigure[\,\,$\lambda=168.5$]{\includegraphics[width=.45\linewidth]{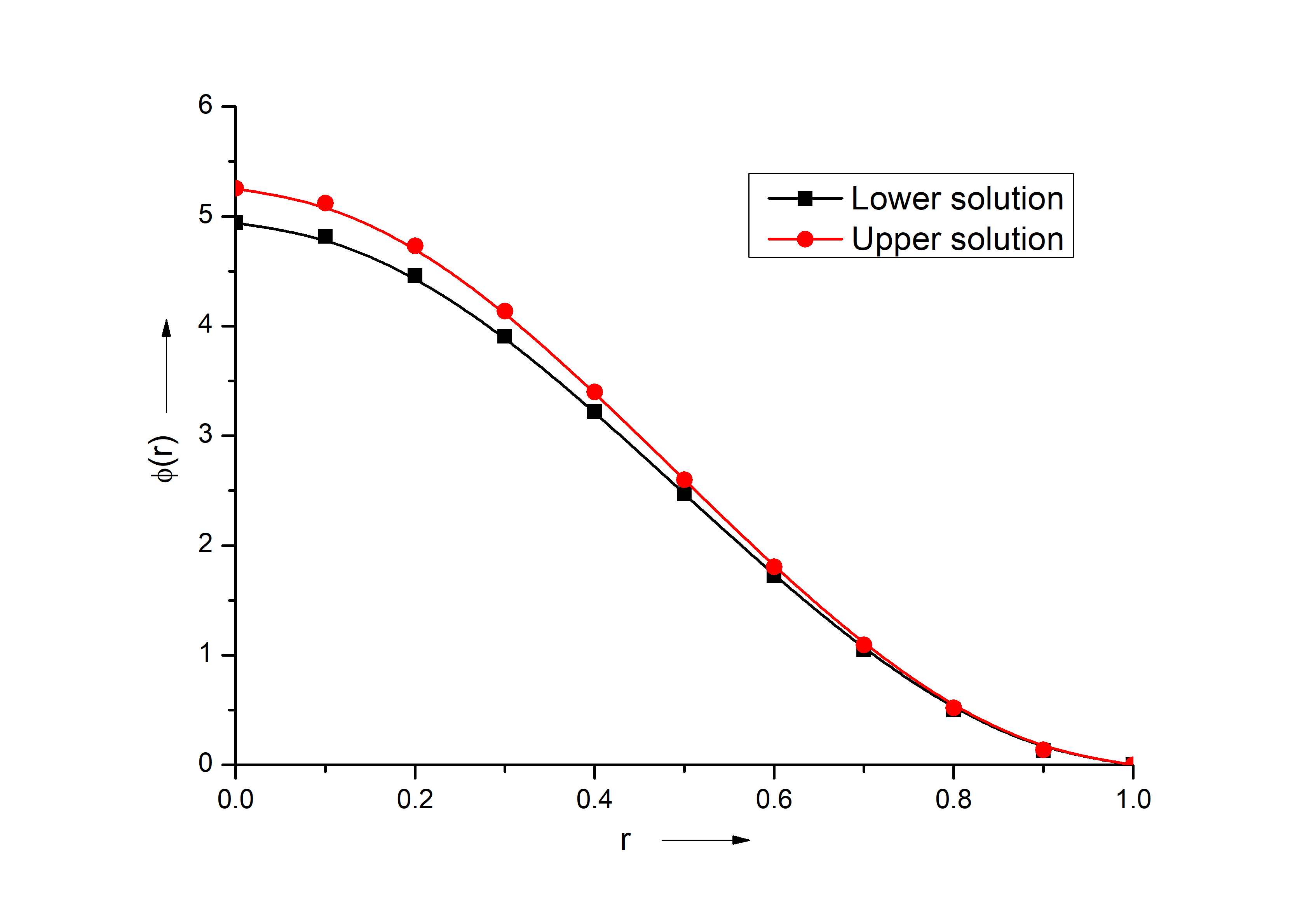}} 
\caption{Graph of $\phi(r)$ versus $r$ for positive $\lambda$.}
\label{P1Figure5}
\end{figure}

\begin{figure}[H]
\centering
\subfigure[\,\,$\lambda=-1$]{\includegraphics[width=.45\linewidth]{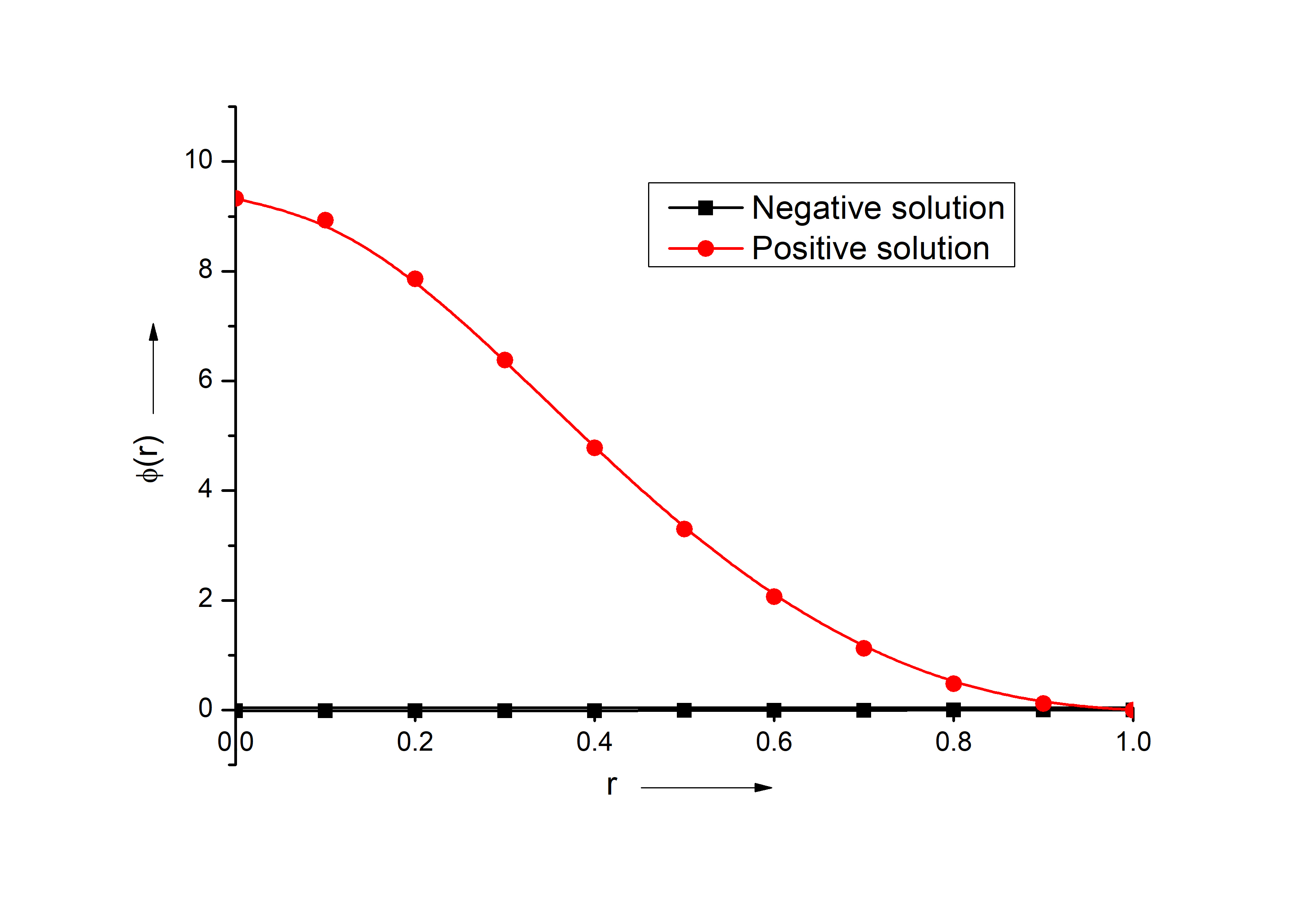}}  
\subfigure[\,\,$\lambda=-10$]{\includegraphics[width=.45\linewidth]{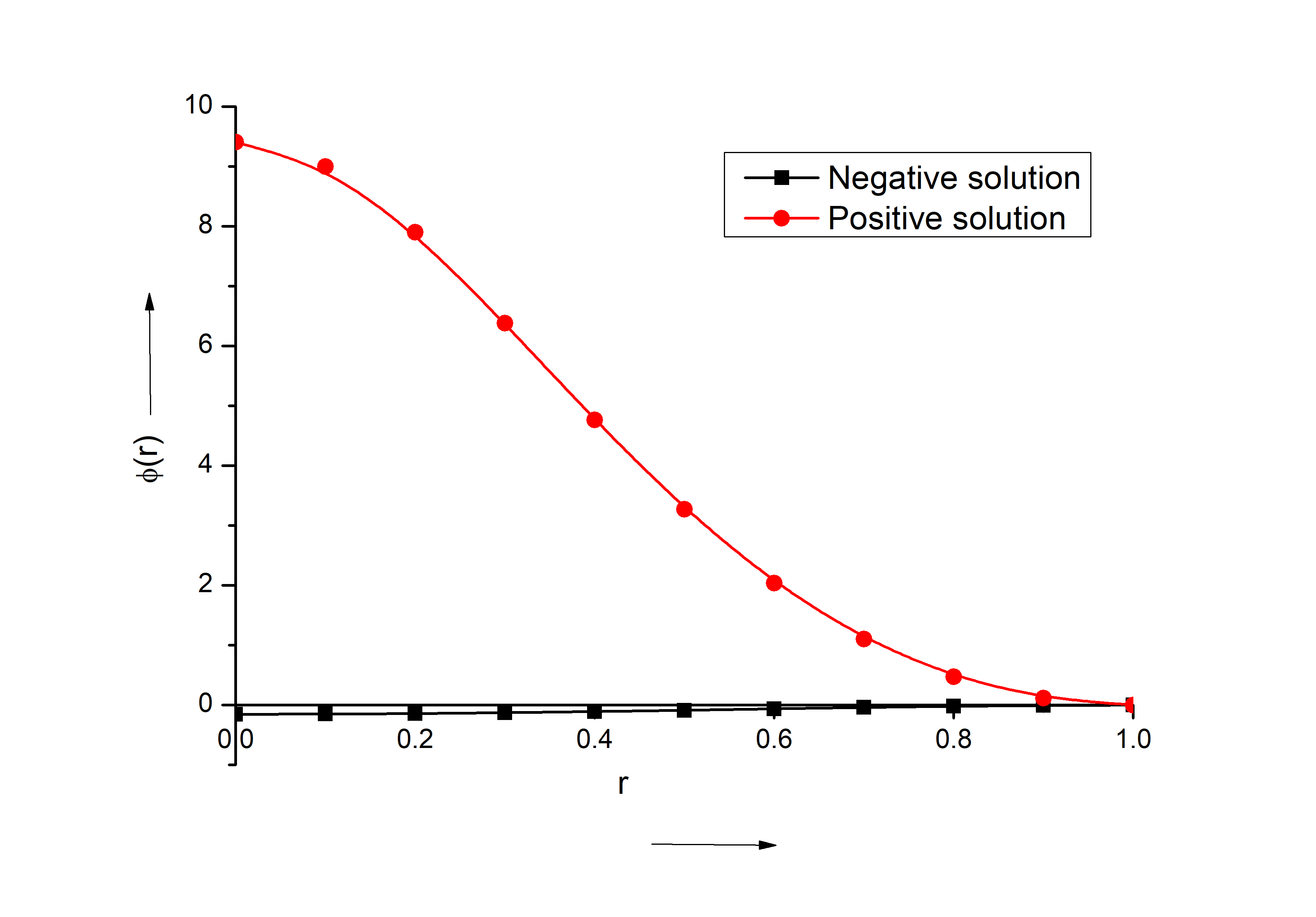}}
\end{figure}
\begin{figure}[H]
\centering
\subfigure[\,\,$\lambda=-15$]{\includegraphics[width=.45\linewidth]{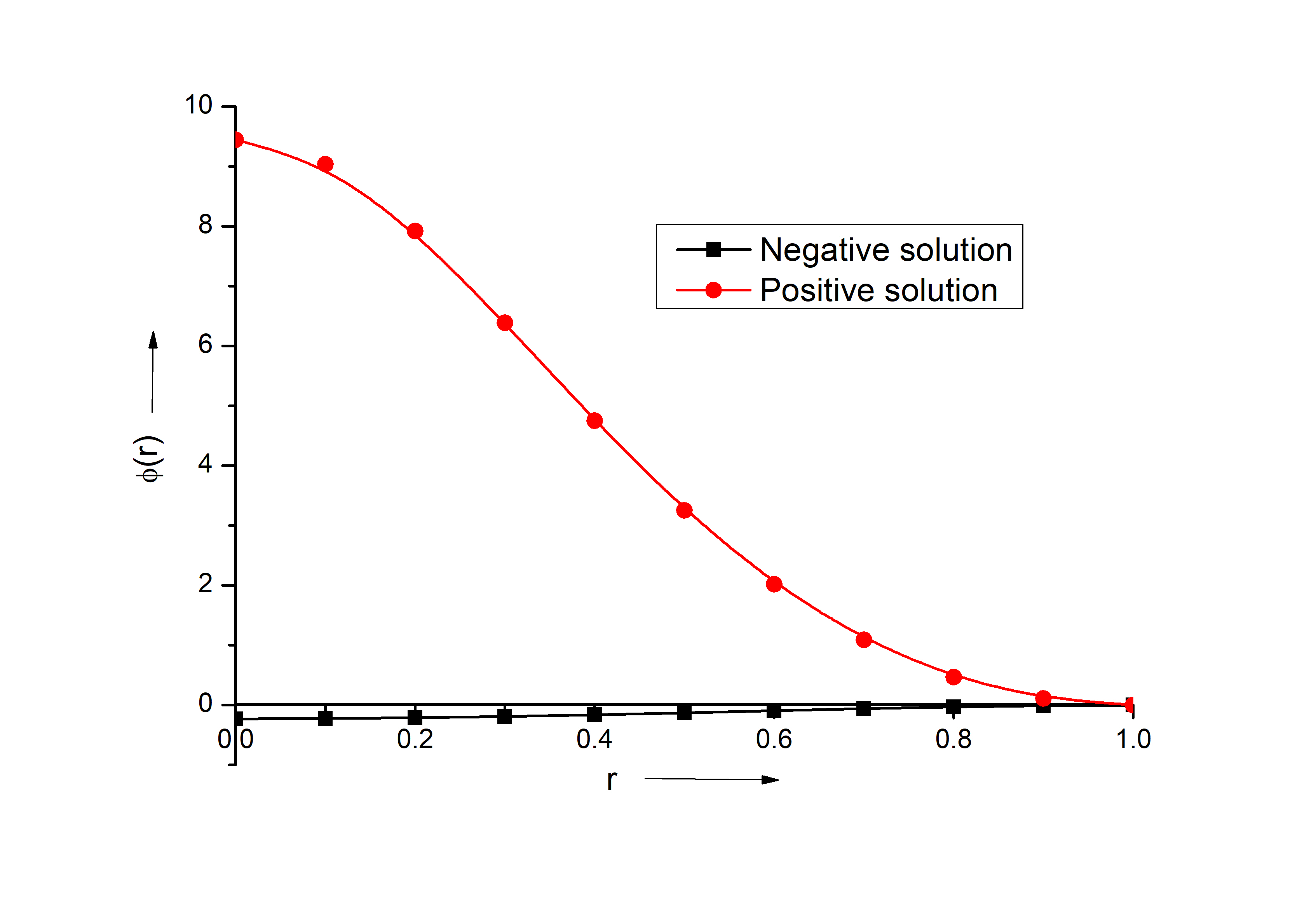}} 
\subfigure[\,\,$\lambda=-25$]{\includegraphics[width=.45\linewidth]{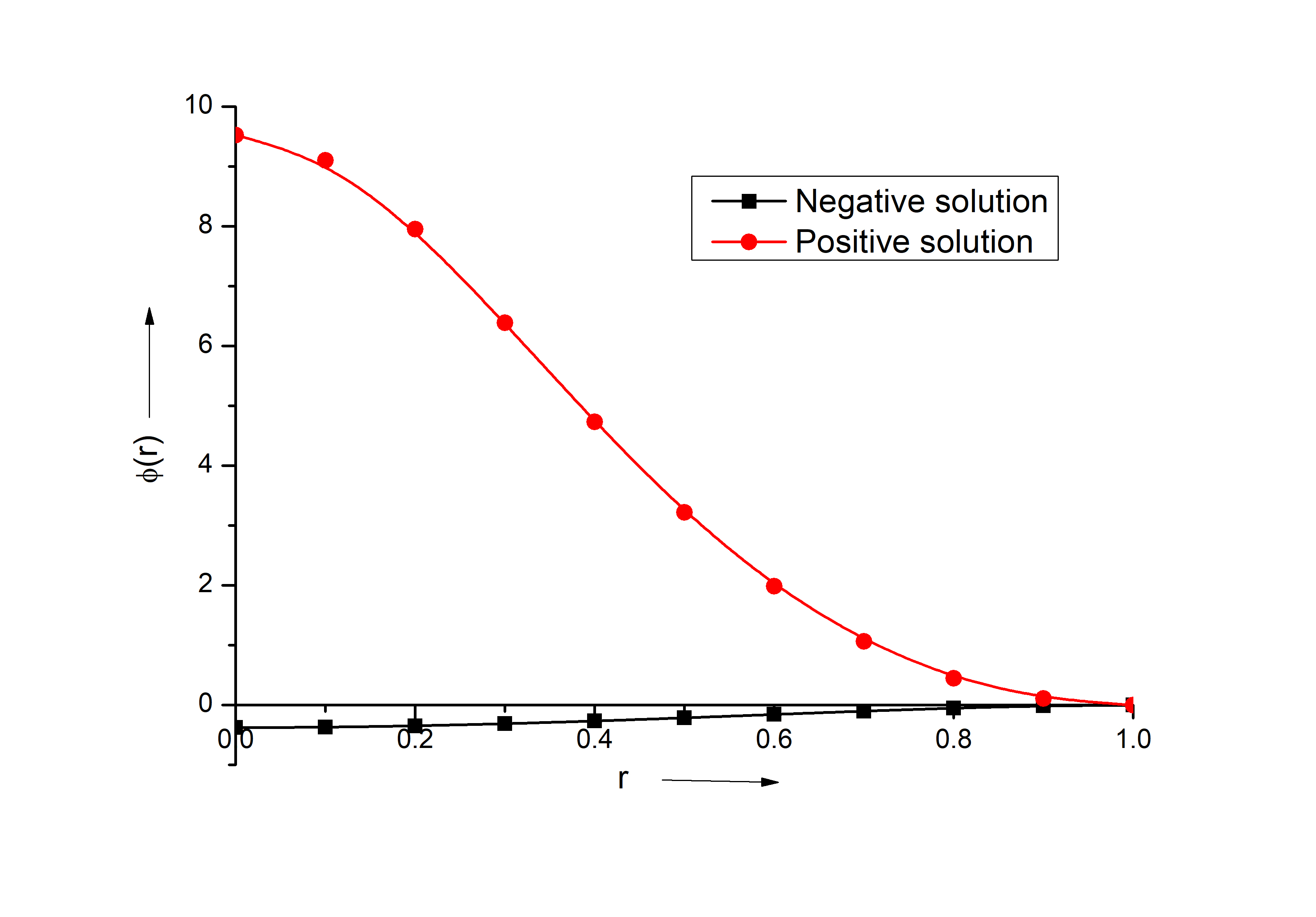}} 
\caption{Graph of $\phi(r)$ versus $r$ for negative $\lambda$.}
\label{P1Figure6}
\end{figure}

\section{Conclusions}\label{conclusions}

In this paper, we have applied an iterative numerical method to nonlinear singular boundary value problems that arise in the theory of epitaxial growth. The proposed technique gives us approximate numerical solutions that are very close to exact solutions of the given differential equation.
We have shown that for negative values of the forcing parameter $\lambda$ two solutions always coexist: they are ordered and one is negative and the other
positive. For more negative values of this parameter the solutions separate. When the value of this parameter is set to zero then one of the solutions becomes
trivial but there still exists a second, nontrivial and positive, solution. For positive values of $\lambda$ below some critical threshold, which we have numerically estimated, there exist two ordered solutions, both nontrivial and positive, which become closer as the value of $\lambda$ increases. We conjecture that both solutions merge into a single solution (still nontrivial and positive) when $\lambda$ is set to its critical value. No solutions were numerically
detected for supercritical values of $\lambda$. The results for the three different sets of boundary conditions we have considered herein show a qualitative
agreement, although they do not behave in the same quantitative way. Also note that the present results agree with and extend those in \cite{CarlosRadial2013},
which were obtained by means of a fourth order Runge-Kutta method.

We expect the present results to have an impact in the understanding of these theoretical models of epitaxial growth. As we have seen the properties of the
solution set depend on the value of the parameter $\lambda$. This parameter has a clear physical meaning: it is the rate at which new material is deposited
onto the epitaxially growing solid. The fact that we take $G(r) \equiv 1$ in equation~\eqref{Eq 3} means that this rate is homogeneous. Our numerical analysis shows that two stationary solutions exists if $\lambda$ is small enough. The lower solution, if the full time dependent model were considered, should be dynamically stable, while the upper solution should be unstable. This means that, if the deposition rate is small enough, one should observe, instead of the system growing, the formation of a stationary mound, at least for suitable initial conditions. The formation of this stationary mound, which may look like a counterintuitive phenomenon in the presence of a constant flux of mass, is possible due to the boundary conditions. We expect both the Navier conditions to be
related to the presence of an open border in the system, that is, the deposited material that gets to the border leaves the system and never gets back into it.
On the other hand, the Dirichlet conditions could be related to a system that is undergoing material drainage on the border, what could explain the quantitative
differences among both sets of boundary conditions, such as the higher critical $\lambda$ for the Dirichlet conditions. In any case, for a large enough
$\lambda$ no stationary solutions still exist, and this is due to the physical fact that the rate at which new material is entered into the system is simply too high and the system will be growing forever. In spite of the simplicity of our mathematical models, we still expect these predictions to be testable against
suitable experiments. Indeed, the existence theory can be related to physical phenomena and therefore the validity of the models could be established, at least at a qualitative level.

\section*{Acknowledgements}
This work has been supported by a grant provided by DST(SERB), New Delhi, India, File no. SB/S4/MS/805/12
and by the Government of Spain (Ministry of Economy, Industry and Competitiveness) through Project MTM2015-72907-EXP.

\bibliography{MasterR}
\end{document}